\documentclass[11pt]{amsart}
\usepackage{latexsym}
\usepackage{amssymb}
\usepackage{amsmath}

\usepackage{amsrefs}
\usepackage{latexsym, enumerate}

\usepackage{amsfonts}

\newtheorem{theorem}{Theorem}[section]
\newtheorem{lemma}[theorem]{Lemma}
\newtheorem{proposition}[theorem]{Proposition}
\newtheorem{corollary}[theorem]{Corollary}
\newtheorem{definition}[theorem]{Definition}
\newtheorem{example}[theorem]{Example}
\newtheorem{remark}[theorem]{Remark}

\newtheorem{question}[theorem]{Question}

\newcommand\esssup{\mathop{\rm esssup}}

\newcommand\tr{\mathop{\rm tr}}

\newcommand\nph{\varphi}

\newcommand\op{\mathop{\rm op}}
\newcommand\dd{\mathop{\rm d}}

\newcommand\eh{\mathop{\rm eh}}
\newcommand\hh{\mathop{\rm h}}
\newcommand\ph{\mathop{\rm ph}}

\newcommand\dom{\mathop{\rm dom}}
\newcommand\Aff{\mathop{\rm Aff}}
\newcommand\Assoc{\mathop{\rm Assoc}}
\newcommand\cb{\mathop{\rm cb}}

\newcommand\mcb{\mathop{\rm M}^{\cb}}

\newcommand{\cl}[1]{\mathcal{#1}}
\newcommand{\bb}[1]{\mathbb{#1}}

\begin{document}

\title{Local operator multipliers and positivity}

\author{N.M. Steen, I.G. Todorov and L. Turowska}

\address{Pure Mathematics Research Centre, Queen's University Belfast,
Belfast BT7 1NN, United Kingdom}

\email{nsteen01@qub.ac.uk}

\address{Pure Mathematics Research Centre, Queen's University Belfast,
Belfast BT7 1NN, United Kingdom}

\email{i.todorov@qub.ac.uk}

\address{Department of Mathematics, Chalmers University of Technology and the University of Gothenburg,
Sweden}

\email{turowska@chalmers.se}


\maketitle

\begin{abstract}
We establish an unbounded version of Stinespring's Theorem and a
lifting result for Stinespring representations of completely
positive modular maps defined on the space of all compact operators.
We apply these results to study positivity for Schur multipliers. We characterise positive local Schur multipliers, and
provide a description of positive local Schur multipliers of Toeplitz type. We introduce local
operator multipliers as a non-commutative analogue of local Schur
multipliers, and characterise them extending both the
characterisation of operator multipliers from \cite{jtt2009} and
that of local Schur multipliers from \cite{stt2011}. We provide a
description of the positive local operator multipliers in terms of
approximation by elements of canonical positive cones.
\end{abstract}

\section{Introduction}\label{s_intro}

A bounded function
$\nph:\mathbb{N}\times\mathbb{N}\rightarrow\mathbb{C}$ is called a
Schur multiplier if $(\nph(i,j)a_{ij})$ is the matrix of a bounded
linear operator on $\ell^2$ whenever $(a_{ij})$ is such. The study
of Schur multipliers was initiated by I. Schur in the early 20th
century, and a characterisation of these objects was given by A.
Grothendieck in his {\it R$\acute{e}$sum$\acute{e}$} \cite{Gro} (see
also \cite{Pi}). A measurable version of Schur multipliers was
developed by M.S. Birman and M.Z. Solomyak (see \cite{BS4} and the
references therein) and V. V. Peller \cite{peller_two_dim}. More
concretely, given standard measure spaces $(X,\mu)$ and $(Y,\nu)$
and a function $\nph : X\times Y \rightarrow \bb{C}$, one defines a
linear transformation $S_{\varphi}$ on the space of all
Hilbert-Schmidt operators from $H_1 = L^2(X,\mu)$ to $H_2 =
L^2(Y,\nu)$ by multiplying their integral kernels by $\varphi$; if
$S_{\varphi}$ is bounded in the operator norm (in which case $\nph$
is called a \emph{measurable Schur multiplier}), one extends it to
the space $\cl K(H_1,H_2)$ of all compact operators from $H_1$ into
$H_2$ by continuity. The map $S_{\varphi}$ is defined on the space
$\cl B(H_1,H_2)$ of all bounded linear operators from $H_1$ into
$H_2$ by taking the second dual of the constructed map on $\cl
K(H_1,H_2)$. A characterisation of measurable Schur multipliers,
extending Grothendieck's result, was obtained in
\cite{peller_two_dim} (see also \cite{spronk}). Namely, a function
$\nph\in L^{\infty}(X\times Y)$ was shown to be a Schur multiplier
if and only if $\nph$ coincides almost everywhere with a function of
the form $\sum_{k=1}^{\infty} a_k(x) b_k(y)$, where $(a_k)_{k\in
\bb{N}}$ and $(b_k)_{k\in \bb{N}}$ are families of essentially
bounded measurable functions such that $\esssup_{x\in X}
\sum_{k=1}^{\infty} |a_k(x)|^2 < \infty$ and $\esssup_{y\in Y}
\sum_{k=1}^{\infty} |b_k(y)|^2 < \infty$.

A local version of Schur multipliers was defined and studied in
\cite{stt2011}. Local Schur multipliers are, in general, unbounded,
but necessarily closable, densely defined linear transformations on
$\cl B(L^2(X,\mu),L^2(Y,\nu))$. A measurable function $\nph :
X\times Y\rightarrow \bb{C}$ was shown in \cite{stt2011} to be a
local Schur multiplier if and only if it agrees almost everywhere
with a function of the form $\sum_{k=1}^{\infty} a_k(x) b_k(y)$,
where $(a_k)_{k\in \bb{N}}$ and $(b_k)_{k\in \bb{N}}$ are families
of measurable functions such that $\sum_{k=1}^{\infty} |a_k(x)|^2 <
\infty$ for almost all $x\in X$ and $\sum_{k=1}^{\infty} |b_k(y)|^2
< \infty$ for almost all $y\in Y$.

In \cite{ks2006}, a quantised version of Schur multipliers, called
universal operator multipliers, was introduced. Universal operator
multipliers are defined as elements of C*-algebras satisfying
certain boundedness conditions, and hence are non-commutative
versions of continuous Schur multipliers. A characterisation of
universal operator multipliers, generalising Grothendieck-Peller's
results, was obtained in \cite{jltt2009}.

In the present paper, we introduce and study local operator
multipliers. Due to their spatial nature, the natural setting here
is that of von Neumann algebras. Pursuing the analogue with the
commutative setting, where local multipliers are measurable (not
necessarily bounded) functions on two variables, we define local
operator multipliers as operators affiliated with the tensor product
of two von Neumann algebras. We characterise local operator
multipliers, extending both the description of local Schur
multipliers from \cite{stt2011} and the description of universal
operator multipliers from \cite{jltt2009}. We further characterise
the positive local Schur multipliers (Section \ref{s_plsm}), as well
as the positive local operator multipliers (Section \ref{s_pos}). We
describe positive local Schur multipliers of Toeplitz type, and
consider local Schur multipliers that are divided differences, that
is, functions of the form $\frac{f(x) - f(y)}{x - y}$. We show that
such a divided difference is a positive local Schur multiplier
with respect to every standard Borel measure if and only if $f$ is an operator monotone function.

Our main tool for characterising positivity of multipliers is an
unbounded version of Stinespring's Theorem (Section
\ref{unbounded_stinespring}). In the literature, there are a number
of versions of Stinespring's Theorem for completely positive, not
necessarily bounded, maps, defined on
*-algebras, see e.g. \cite{joita}, \cite{powers} and \cite{schudgen}. Our version
differs from the existing ones in that the domain is a non-unital
pre-C*-algebra, and a partial boundedness of the map is assumed --
as a result, we are able to obtain more specific conclusions
regarding the (closable) operator implementing the Stinespring
representation. In Section~ \ref{s_rtst}, we prove a lifting result
for Stinespring representations of completely positive maps defined
on the space of compact operators (Theorem \ref{p_compl2}). The
result, which we believe is interesting in its own right, is used in
Section~ \ref{s_plsm} to obtain a lifting result for positive Schur
multipliers, and provides an alternative approach to the unbounded
Stinespring Theorem from Section~ \ref{unbounded_stinespring}.

\medskip

All Hilbert spaces appearing in the paper will be assumed to be
separable. The inner product of a Hilbert space $H$ is denoted by
$(\cdot,\cdot)_H$, if $H$ needs to be emphasised. We let $I_H$
denote the identity operator acting on $H$, and write $I$ when $H$
is clear from the context. For Hilbert spaces $H$ and $K$, we denote
by $\cl B(H,K)$ (resp. $\cl K(H,K)$, $\cl C_2(H,K)$) the space of
all bounded linear (resp. compact, Hilbert-Schmidt) operators from
$H$ into $K$, and let $\cl B(H) = \cl B(H,H)$, $\cl K(H) = \cl
K(H,H)$ and $\cl C_2(H) = \cl C_2(H,H)$. We denote by $I_H$ the
identity operator acting on $H$. The operator norm is denoted by
$\|\cdot\|_{\op}$. We often use the weak* topology of $\cl B(H,K)$,
which arises from the identification of this space with the dual of
the space of all nuclear operators from $K$ into $H$. The weak*
continuous linear maps on $\cl B(H,K)$ will be referred to as normal
maps. If $\alpha$ is a cardinal number, we let $H^{\alpha}$ denote
the direct sum of $\alpha$ copies of $H$, and for $x\in \cl B(H)$,
we let $x\otimes 1_{\alpha}$ be the ampliation of $x$ acting on
$H^{\alpha}$. We let $\ell^2_{\alpha}$ be the Hilbert space of
square summable sequences of length $\alpha$.

Throughout the paper, we will use notions and results from Operator
Space Theory; we refer the reader to the monographs \cite{blm},
\cite{er}, \cite{paulsen} and \cite{pisier_intr}. If $\cl A\subseteq
\cl B(H)$ is a C*-algebra, we denote by $\cl A'$ the commutant of
$\cl A$, and by $M_{n,m}(\cl A)$ the set of all $n\times m$ matrices
with entries in $\cl A$ which define bounded operators (here $n$ or
$m$ may be $\infty$). For a matrix $a\in M_{n,m}(\cl A)$, we denote
by $a^t$ its transposed matrix. If $\cl M\subseteq \cl B(H)$ is a
von Neumann algebra, we will denote by $\Aff\mathcal{M}$ the set of
all densely defined operators on $H$ that are affiliated with $\cl
M$; thus, $T\in\Aff\mathcal{M}$ if and only if the spectral measure
of the operator $|T| = (T^*T)^{1/2}$ takes values in $\cl M$ and the partial isometry in the polar
decomposition of $T$ belongs to $\cl M$. The domain of an (unbounded) operator $T$ will be denoted
by $\dom(T)$.

If $\cl A$ and $\cl B$ are linear spaces, we will denote by $\cl
A\odot \cl B$ their algebraic tensor product; if $\cl A$ and $\cl B$
are von Neumann algebras, their weak* spatial tensor product will be
denoted by $\cl A\bar\otimes\cl B$. The linear span of a subset $\cl
X$ of a vector space will be denoted by $[\cl X]$.

\section{An unbounded version of Stinespring's theorem}\label{unbounded_stinespring}

The classical Stinespring's Representation Theorem for completely
positive maps states that if $\mathcal{A}$ is a unital C*-algebra
and $\Phi:\mathcal{A}\rightarrow\mathcal{B}(H)$ is a completely
positive map, then there exists a Hilbert space $K$, a unital
*-homomorphism $\pi:\mathcal{A}\rightarrow\mathcal{B}(K)$ and a
bounded operator $V:H\rightarrow K$ with $\|\Phi(1)\|=\|V\|^2$ such
that $\Phi(a)=V^*\pi(a)V$, $a\in \cl A$. In the case where
$\overline{[\pi(\mathcal{A})VH]}=K$, we say that $(\pi, V, K)$ is a
\emph{minimal Stinespring representation} for $\Phi$ (see
\cite{paulsen}). Our aim in this section is to prove an unbounded
version of Stinespring's Theorem for maps defined on
pre-C*-algebras, and apply it in the special case where the
C*-completion of the domain coincides with the C*-algebra of all
compact operators acting on a Hilbert space.

Let $\cl A$ be a $*$-algebra, $\cl X$ be a linear (not necessarily
closed) subspace of a Hilbert space $H$ and $\cl L(\cl X)$ be
the space of all linear mappings on $\cl X$. A linear mapping $\Phi
: \cl A\to\cl L(\cl X)$ will be called {\it completely positive} if
$$\sum_{k,l=1}^n(\Phi(a_k^*a_l)\xi_l,\xi_k)\geq 0$$
for arbitrary $n\in \bb{N}$, $a_1,\ldots,a_n\in \cl A$ and
$\xi_1,\ldots\xi_n\in\cl X$.

We denote by $M(\cl B)$ the multiplier algebra of a C*-algebra $\cl
B$. We recall that, if $\cl B$ is identified with a subalgebra of
its enveloping von Neumann algebra (that is, its second dual) $\cl
B^{**}$ then $M(\cl B)\cong \{x\in \cl B^{**} : x\cl B\subseteq \cl
B, \cl B x\subseteq \cl B\}$.

\begin{theorem}\label{th_b}
Let $\cl B$ be a C*-algebra and $\cl A\subseteq \cl B$ be a dense
*-subalgebra of the form $\cl A = \cup_{k=1}^{\infty} p_k \cl A p_k$,
where $(p_k)_{k\in \bb{N}}\subseteq M(\cl B)$ is an increasing
sequence of projections with $p_n\cl A \subseteq \cl A$ and $\cl A
p_n \subseteq \cl A$. Let $H$ be a Hilbert space, $(q_k)_{k\in
\bb{N}}$ be an increasing sequence of projections on $H$ with strong
limit $I$, and $\cl X = \cup_{k\in \bb{N}} q_k H$. Assume that $\Phi
: \cl A\to\cl L(\cl X)$ is a completely positive map such that
$\Phi(p_k ap_m) = q_k\Phi(a)q_m$, $k,m\in \bb{N}$, $a\in \cl A$. The
following are equivalent:

(i) \ the restriction $\Phi|_{p_k\cl A
p_k}$ is bounded for each $k\in \bb{N}$;

(ii) there exist a Hilbert space $K$, a bounded *-representation
$\pi : \cl A\rightarrow \cl B(K)$ and an operator $V : \cl X
\rightarrow K$, such that $V|_{q_k H}$ is bounded for every $k\in
\bb{N}$, and
$$(\Phi(a)\xi,\eta) = (\pi(a)V\xi,V\eta), \ \ \ a\in \cl A, \ \xi,\eta \in \cl X.$$

Moreover, the operator $V$ appearing in (ii) can be chosen to be closable.
\end{theorem}
\begin{proof}
(i)$\Rightarrow$(ii) Let $\cl A_k = p_k\cl A p_k$, $H_k = q_k H$,
and write $\Phi_k$ for the restriction of $\Phi$ to $\cl A_k$. Since
$\Phi_k$ is bounded, it can be extended to a bounded map from $\cl
B_k \stackrel{def}{=} \overline{\cl A_k}$ into $\cl B(H_k)$ (where
the closure is taken in the norm topology of $\cl B$), which we also
denote by $\Phi_k$. Let $\Phi_k^{**} : \cl B_k^{**} \rightarrow \cl
B(H_k)^{**}$ be the second dual of $\Phi_k$ and $\cl E_k : \cl
B(H_k)^{**} \rightarrow \cl B(H_k)$ be the canonical projection
(we point out that $\cl E_k$ is the dual of the inclusion of $\cl B(H_k)_*$ into
$\cl B(H_k)^{*}$ and is hence weak* continuous). We note that if
$k\leq m$ then $p_k\in M(\cl B_m)\subseteq \cl B_m^{**}$. Set
$\widetilde{\Phi}_k = \cl E_k\circ \Phi_k^{**}$; thus,
$\widetilde{\Phi}_k$ is a weak* continuous completely positive map
from $\cl B_k^{**}$ into $\cl B(H_k)$ extending $\Phi_k$. We note
that if $k \leq m$ then $\widetilde{\Phi}_m|_{\cl B_k^{**}} =
\widetilde{\Phi}_k$. By weak* continuity, $\widetilde{\Phi}_k (p_l x
p_m) = q_l \widetilde{\Phi}_k(x) q_m$ whenever $l,m \leq k$ and
$x\in \cl B_k^{**}$.

We now modify the well-known Stinespring construction. Let $\cl L$
be the linear space generated by $\cl A\otimes\cl X$ and
all vectors of the form $1\otimes\xi$, $\xi\in\cl X$. We define a
sesquilinear form $\langle\cdot,\cdot\rangle_1$ on $\cl L$: if
$\zeta = \sum_{i=1}^s a_i\otimes\xi_i$ and $\theta = \sum_{j=1}^t
b_j\otimes\eta_j$, set
$$\langle\theta,\zeta\rangle_1 = \sum_{i=1}^s \sum_{j=1}^{t}(\widetilde{\Phi}_k(p_ka_i^*b_jp_k)\eta_j,\xi_i),$$
where $k$ is such that $\xi_i$, $\eta_j\in H_k$, for all
$i=1,\ldots, s$ and all $j = 1,\dots,t$. We note that $\langle\cdot,\cdot\rangle_1$ is
well-defined; indeed, if $\xi_i$ and $\eta_j\in H_m$ for some $m$,
assuming that $m \leq k$, we have
\begin{eqnarray*}
(\widetilde{\Phi}_k(p_ka_i^*b_jp_k)\eta_j,\xi_i)&=&(q_m\widetilde{\Phi}_k(p_ka_i^*b_jp_k)q_m\eta_j,\xi_i)\\
&=&
(\widetilde{\Phi}_k(p_kp_ma_i^*b_jp_mp_k)\eta_j,\xi_i) = (\widetilde{\Phi}_k(p_ma_i^*b_jp_m)\eta_j,\xi_i)\\
&=& (\widetilde{\Phi}_m(p_ma_i^*b_jp_m)\eta_j,\xi_i).
\end{eqnarray*}
Since $\Phi$ is completely positive, $\langle\cdot,\cdot\rangle_1$ is a semidefinite
inner product on $\cl L$.

Let $N := \{ \zeta\in\cl L:\langle\zeta,\zeta\rangle_1=0\}$ and let
$$\rho(a)\left(\sum_{k=1}^na_k\otimes\nph_k\right) = \sum_{k=1}^n aa_k\otimes\nph_k, \ a\in \cl A;$$
$\rho$ is a well-defined $*$-homomorphism. A standard application of
the Cauchy-Schwartz inequality shows that $N$ is invariant under
$\rho(\cl A)$. Let $i:\cl L\to\cl L/N$ be the quotient map. Then
$(i(\xi),i(\eta))\stackrel{def}{=}\langle \xi,\eta\rangle_1$ defines
a scalar product on $\cl L/N$ and $\pi(a)i(\eta)=i(\rho(a)\eta)$ is
a well-defined $*$-representation  of $\cl A$ on $\cl L/N$. Let $\cl
H$ be the Hilbert space completion of $\cl L/N$.

We claim that $\pi$ is bounded. In fact, for $\zeta =
\sum_{i=1}^na_i\otimes \xi_i\in \cl L$, there exists $k$ such that
$\xi_i\in H_k$ and if $a_i\ne 1$, $a_i\in \cl A_k$, for all
$i=1,\ldots, n$. Since $\widetilde{\Phi}_k$ is completely
positive, we obtain, for $a\in \cl A_k$,
\begin{eqnarray*}
\|\pi(a)i(\zeta)\|^2&=&\langle\rho(a)\zeta,\rho(a)\zeta\rangle_1=
\left\langle\sum_{i=1}^naa_i\otimes\xi_i,\sum_{i=1}^naa_i\otimes\xi_i\right\rangle_1\\
& = & \sum_{i,j=1}^n(\widetilde{\Phi}_k(p_ka_i^*a^*aa_jp_k)\xi_i,\xi_j)\\
& \leq & \|a^*a\|\sum_{i,j=1}^n(\widetilde{\Phi}_k(p_ka_j^*a_ip_k)\xi_i,\xi_j)
\\
& = &
\|a\|^2\left\langle\sum_{i=1}^na_i\otimes\xi_i,\sum_{i=1}^na_i\otimes\xi_i\right\rangle_1
= \|a\|^2\|i(\zeta)\|^2,
\end{eqnarray*}
giving the statement.

Define $V : \cl X\to \cl H$ by $V\xi = i(1\otimes\xi)$. If
$\xi,\eta\in H_k$ then
$$|(V\xi,\eta)| = |(\widetilde{\Phi}_k(p_k) \xi, \eta)| \leq \|\Phi_k\|\|\xi\|\|\eta\|,$$
showing that $V|_{H_k}$ is bounded for each $k$. Moreover,
\begin{eqnarray}\label{pres}
(\Phi(a)\zeta,\theta) =
\langle\rho(a)(1\otimes\zeta),1\otimes\theta\rangle_1=(\pi(a)i(1\otimes\zeta),i(1\otimes\theta))
= (\pi(a) V\zeta, V\theta).\nonumber
\end{eqnarray}

We show that $V$ is closable. If $\eta_n\to 0$ and
$i(1\otimes\eta_n)\to  f$ then, for each $a\otimes\xi\in\cl L$ with $a\in \cl A\cup\{1\}$, $\xi\in \cup_{k=1}^{\infty} H_k$,
we have that
$$(i(a\otimes\xi),i(1\otimes\eta_n))\to (i(a\otimes\xi),f).$$ On the other hand,
if $\xi\in H_k$ and $\eta_n \in H_{l_n}$, $l_n \geq k$, then
\begin{eqnarray*}
(i(a\otimes\xi),i(1\otimes\eta_n)) & = &
(\widetilde{\Phi}_{l_n}(p_{l_n}a p_{l_n})\xi,\eta_n) =
(\widetilde{\Phi}_{l_n}(p_k a p_k)\xi,\eta_n)\\
& = & (\widetilde{\Phi}_{k}(p_k a p_k)q_k\xi,q_k\eta_n) \rightarrow_{n\rightarrow \infty}
0.
\end{eqnarray*}
It follows that $(i(a\otimes\xi),f) = 0$, for all $a\otimes\xi \in
\cl L$. As $i(\cl L)$ is dense in $\cl H$, we conclude that $f=0$.

(ii)$\Rightarrow$(i) is trivial.
\end{proof}

\begin{theorem}\label{c_c2}
Let $H$ be a Hilbert space, $(p_n)_{n=1}^\infty$ be an increasing sequence of projections on $H$
such that $\vee_{n=1}^\infty p_n=I$ and  let $H_n=p_nH$, $n\in \bb{N}$. Let $\cl A
:= \cup_{n=1}^\infty \cl C_2(H_n)$ and  $\cl X=\cup_{n=1}^\infty
H_n$. Assume that $\Phi : \cl A\to\cl L(\cl X)$ is a completely
positive map such that $\Phi(p_n x p_m)=p_n\Phi(x)p_m$ for all $n,
m\in \bb{N}$, and $\Phi_n=\Phi|_{\cl C_2(H_n)}$ is bounded with
respect to the operator norm on $\cl B(H_n)$.

(i) There exists a family $(V_i)_{i=1}^\infty$ of closable linear
maps from $\cl X$ into $H$ such that
\begin{equation}\label{eq_su}
(\Phi(a)\xi,\eta)=\sum_{i=1}^\infty(aV_i\xi,V_i\eta), \ \ a\in \cl
A, \ \ \xi,\eta\in \cl X.
\end{equation}

(ii) If $\cl D\subseteq \cl B(H)$ is a unital $C^*$-subalgebra,
$(p_n)_{n=1}^\infty\subseteq\cl D$ and $\Phi$ is $\cl D$-bimodular,
then $V_i$, $i\in{\mathbb N}$, can be chosen to be closable
operators affiliated with $\cl D'$.
\end{theorem}
\begin{proof}
In the notation of Theorem \ref{th_b}, $\cl B = \cl K(H)$, and hence the representation
$\pi$ arising in Theorem \ref{th_b} is unitarily equivalent to an
ampliation of the identity representation. At the expense of
changing the operator $V$, we may thus assume that $\pi(a) =
a\otimes 1$, $a\in \cup_{k=1}^{\infty} \cl C_2(H_k)$. Let
$$V=(V_1,\ldots,V_n,\ldots)^t$$ be the corresponding matrix of $V$;
identity (\ref{eq_su}) follows now trivially from Theorem
\ref{th_b}. The fact that the operators $V_i$ are closable follows
easily from the closability of $V$.

We show that we can choose $V_i$  to be  affiliated with $\cl D'$
when $\Phi$ satisfies condition (ii). The arguments that follow are
similar to the ones in
\mbox{\cite[Corollary~5.9]{christensen-sinclair}.} The modularity of
$\Phi$ gives
$$((a\otimes 1)V\xi,(r\otimes 1)V\eta)=((a\otimes 1)V\xi,Vr\eta)
\quad\text{for all } r\in \cl D, a\in \cl A, \xi,\eta\in \cl X.$$
Let $e_a$ be the
projection onto $\overline{(a\otimes 1)V\cl X}$, $a\in \cl A$. Then
$e_a(r\otimes 1)V\eta=e_aVr\eta$, $r\in \cl D$, $\eta\in\cl X$. If
$e = \vee_{a\in \cl A} e_a$ then
$$e(r\otimes 1)V\eta=eVr\eta,\quad r\in \cl D, \eta\in\cl X.$$

As $(r\otimes 1)(a\otimes 1)V\eta=(ra\otimes 1)V\eta$ and $ra\in \cl
A$ ($r\in \cl D$, $a\in \cl A$), we obtain that $e(r\otimes 1)e=(r\otimes 1)e$ for each $r\in \cl D$ and hence $e$ commutes with
$r\otimes 1$, $r\in \cl D$.
This implies
$$(r\otimes 1)(eV)\eta=(eV)r\eta,\quad \eta\in\cl X, \ r\in \cl D,$$
and
$$(\Phi(a)\xi,\eta)=((a\otimes 1)eV\xi,eV\eta), \quad \xi, \eta\in\cl X.$$
The only thing that is left to prove is that if
$eV=(V_1',\ldots,V_n',\ldots)^t$ then $V_i'$ is affiliated with $\cl
D'$, for every $i$.

We have $rV_i'\eta=V_i'r\eta$, for all $r\in \cl D$ and  $\eta\in\cl
X$. The operator $V_i'$ is closable as an operator defined on $\cl
X$; let us denote its closure again by $V_i'$. Then clearly $r\cl
\dom(V_i')\subseteq\dom(V_i')$ and $rV_i'=V_i'r$. As $V_i'$ are
closable, the equality holds for each $r$ in the strong closure of
${\cl D}$. Since $\vee_{n=1}^\infty p_n=I$ and $p_n\in\cl D$, the
C*-algebra $\cl D$ is non-degenerate and hence $rV_i'=V_i'r$ for all
$r\in \cl D''$. Thus, $V_i'$ is affiliated
with $\cl D'$.
\end{proof}

\section{Lifting for Stinespring's representations}\label{s_rtst}

Our aim in this section is to obtain a lifting for Stinepring's
representations of maps defined on the algebra of compact operators.
The result will be used in Section \ref{s_plsm} to obtain a lifting
for positive Schur multipliers, but we believe that it is
interesting in its own right as well. It also provides an
independent route to Theorem \ref{c_c2}.

Let $H$ be a Hilbert space and
$\Phi : \cl K(H)\to \cl B(H)$ be a completely positive map.
Then there exists a Stinespring representation
\begin{equation}\label{eq_sp}
\Phi(x)=V^*(x\otimes 1_\alpha)V=\sum_{i=1}^\alpha a_i^*xa_i, \ \ x\in \cl K(H),
\end{equation}
where $1\leq \alpha\leq \infty$ and $V =(a_1,a_2,\dots)^t\in
M_{\alpha,1}(\cl B(H))$. We will say in this case that $V$
\emph{implements} the representation (\ref{eq_sp}). If $\Phi$ is
moreover modular over a non-degenerate C*-algebra $\cl A\subseteq
\cl B(H)$, then the entries of $V$ can be chosen from $\cl D'$.

We start by characterising the minimal representations of the map
$\Phi$ in terms of the operator $V$ (Lemma \ref{equiv}). Note that
given any element $(\lambda_i)_{i=1}^{\alpha}\in \ell^2_{\alpha}$, the series
$\sum_{i=1}^{\alpha}\lambda_i a_i$ is norm convergent in
$\mathcal{B}(H)$.  Following S. D. Allen, A. M. Sinclair and R. R.
Smith \cite{ass1993}, we say that the set $\{a_i\}_{i=1}^{\alpha}$
is \emph{strongly independent} if $\sum_{i=1}^{\alpha}\lambda_i a_i
= 0$ implies that $\lambda_i=0$, for all $i$. It was established
in \cite[Lemma 2.2]{ass1993} that, for the case $\alpha = \infty$, the family
$\{a_i\}_{i=1}^{\alpha}$ is strongly independent if and only if the set
$\mathcal{K}=\{(\omega(a_1),\omega(a_2),\dots):\omega\in\mathcal{B}(H)^*\}$
is norm dense in $\ell^2_{\alpha}$.

\begin{lemma}\label{equiv}
Let $V = (a_i)_{i=1}^{\alpha} \in M_{\alpha,1}(\cl B(H))$ implement a representation of the
completely positive map $\Phi : \mathcal{K}(H)\rightarrow \mathcal{B}(H)$,
where $\alpha$ is an at most countable cardinal. The following are equivalent:

(i) \ \ The operator $V$ implements a minimal representation of $\Phi$;

(ii) \ The set
\begin{eqnarray*}\mathcal{F} & = & \left\{(\omega(a_1), \omega(a_2),\dots): \omega \text{ is a
vector functional on } \mathcal{B}(H)\right\}
\end{eqnarray*}
has norm dense linear span in $\ell^2_{\alpha}$;

(iii) The set $\mathcal{J}=\{(\omega(a_1),\omega(a_2),\dots):
\omega\in\mathcal{B}(H)_*\}$ is norm dense in $\ell^2_{\alpha}$;

(iv) \ The set $\{a_i\}_{i=1}^{\alpha}$ is strongly independent.
\end{lemma}
\begin{proof}
(i)$\Rightarrow$(ii) Suppose that the
linear span of $\mathcal{F}$ is not dense and let
$0 \neq (\lambda_i)_{i=1}^{\alpha}\in\mathcal{F}^{\perp}$.
Then
\begin{eqnarray*}
0=\sum_{i=1}^{\alpha}\overline{\lambda_{i}} (a_i\xi,
x^*\eta)\ =\sum_{i=1}^{\alpha} (x a_i\xi,
\lambda_{i}\eta),\ \ \ \ \xi,\eta\in H, \ x\in
\mathcal{K}(H).
\end{eqnarray*} Thus, if $\eta$ is a non-zero vector
in $H$, then $(\lambda_i\eta)_{i=1}^{\alpha}\in H^{\alpha}$ is
non-zero and orthogonal to $(x a_i\xi)_{i=1}^{\alpha}$ for all $\xi$
in $H$ and all $x$ in $\mathcal{K}(H)$.  As a result, the
representation implemented by $V$ cannot be minimal.

(ii)$\Rightarrow$(i) Suppose that the representation is not minimal.
Then $\mathcal{E}=\overline{(\mathcal{K}(H)\otimes 1)VH} \neq
H^{\alpha}$. Let $Q$ is the projection from $H^{\alpha}$ onto
$\mathcal{E}$; since $\cl E$ is invariant under
$\mathcal{K}(H)\otimes 1$, we have that $Q\in(\mathcal{K}(H)\otimes
1)'$ and hence $Q = I_H\otimes Q_0$, where $Q_0$ is a projection in
$\mathcal{B}(\ell^2_{\alpha})$.  Thus we have that
$\mathcal{E}^{\perp} = H\otimes (Q_0^{\perp}\ell^2_{\alpha})$.
Choose a non-zero $(\lambda_i)_{i=1}^{\alpha}\in
Q_0^{\perp}\ell^2_{\alpha}$. Then for every $\eta\in H$ we have
$\eta\otimes (\lambda_i)_{i=1}^{\alpha}\in\mathcal{E}^{\perp}$ and
hence
\[0=\sum_{i=1}^{\alpha}
(x^*a_i\xi,\lambda_i\eta) = \sum_{i=1}^{\alpha}\overline{\lambda_i}
(a_i\xi,x\eta) =
\left(\sum_{i=1}^\alpha\overline{\lambda_i}a_i\xi,x\eta\right),\] for all
$x\in \mathcal{K}(H)$ and all $\xi\in H$. It follows that
$\sum_{i=1}^{\alpha}\overline{\lambda_i} (a_i\xi,\eta) = 0$ for
every $\eta\in H$. Hence $(\overline{\lambda_i})_{i=1}^{\alpha}$ is
orthogonal to $((a_i\xi,\eta))_{i=1}^{\alpha}$, for all $\xi,\eta\in
H$. It follows that the linear span of $\mathcal{F}$ is not dense in
$\ell^2_{\alpha}$.

(ii)$\Rightarrow$ (iii) is trivial.

(iii)$\Rightarrow$(ii) follows from the inclusion
$\mathcal{J}\subseteq\overline{[\mathcal{F}]}^{\|\cdot\|}$ whose
verification is straightforward.

(iii)$\Leftrightarrow$(iv)
The set $\mathcal{J}$ is not dense in $\ell^2_{\alpha}$ if and only if
there exists a non-zero element $(\lambda_i)_{i=1}^{\alpha}\in
\ell^2_{\alpha}$ lying in the orthogonal complement of $\mathcal{J}$;
that is, such that
\[\omega\left(\sum_{i=1}^{\alpha}\lambda_ia_i\right)=\sum_{i=1}^{\alpha}\lambda_{i}\omega(a_{i})=0,
\ \  \  \omega\in\mathcal{B}(H)_*.\] This is equivalent to the
existence of a non-zero $(\lambda_i)_{i=1}^{\alpha}\in
\ell^2_{\alpha}$ such that $\sum_{i=1}^{\alpha}\lambda_i a_i = 0$,
that is, to the $\{a_i\}_{i=1}^{\alpha}$ not being strongly
independent.
\end{proof}

\begin{lemma}\label{l_minimalscalar}
Let $H$ be a separable Hilbert space,
$\Psi:\mathcal{K}(H)\rightarrow\mathcal{B}(H)$ be a completely
positive map, and suppose that $A\in M_{\infty,1}(\cl B(H))$ implements
a representation of $\Psi$. Then there exists $\Lambda\in \cl
B(\ell^2,\ell^2_{\alpha})$, such that the operator
$(I_H\otimes\Lambda)A$ implements a minimal representation of
$\Psi$.
\end{lemma}
\begin{proof}
Let $\mathcal{E}=\overline{(\mathcal{K}(H)\otimes 1)AH}$.  As in the
proof of Lemma \ref{equiv}, the projection $Q$ from $H^{\infty}$
onto $\mathcal{E}$ is of the form $I_H\otimes Q_0$, where $Q_0$ is a
projection in $\mathcal{B}(\ell^2)$. Since $\mathcal{E}$ is reducing
for $\mathcal{K}(H)\otimes 1$, the map
$\rho:\mathcal{K}(H)\rightarrow\mathcal{B}(\mathcal{E})$ (resp.
$\rho' : \mathcal{K}(H)\rightarrow\mathcal{B}(\mathcal{E}^{\perp})$)
given by $\rho(x)=x\otimes 1|_{\mathcal{E}}$ (resp.
$\rho^{\prime}(x)=x\otimes 1|_{\mathcal{E}^{\perp}}$) is a
*-representation of $\mathcal{K}(H)$. Thus, there exists an (at most
countable) cardinal $\alpha$ and a unitary operator
$S:\mathcal{E}\rightarrow H^{\alpha}$ such that
$\rho(x)=S^*(x\otimes 1_{\alpha})S$.  Consider the operator
$T:H^{\infty}\rightarrow H^{\alpha}$ given by $T\xi = SQ\xi$, $\xi
\in H^{\infty}$. Then
\[T(x\otimes 1)T^*=\left(
                     \begin{array}{cc}
                       S & 0 \\
                     \end{array}
                   \right)\left(
                            \begin{array}{cc}
                              \rho(x) & 0 \\
                              0 & \rho^{\prime}(x) \\
                            \end{array}
                          \right)\left(
                                   \begin{array}{c}
                                     S^* \\
                                     0 \\
                                   \end{array}
                                 \right)=S\rho(x)S^*=x\otimes1_{\alpha}.\]
In addition,
\[T^*T=\left(\begin{array}{c}S^*
\\0
\\\end{array}\right)
                                   \left(\begin{array}{cc}S & 0 \\\end{array}\right)
                                   =\left(\begin{array}{cc}S^*S&0\\0&0\\ \end{array}\right)
                                   =\left(\begin{array}{cc}I_{\cl E}&0\\0&0\\ \end{array}\right)
                                   =I_H\otimes Q_0.\]
                                 Thus, $T(x\otimes 1)(I_H\otimes
                                 Q_0)=(x\otimes 1_{\alpha})T$.
Also,
                                 \[T(x\otimes 1)(I_H\otimes
                                 Q_0^{\perp})=\left(
                                                \begin{array}{cc}
                                                  S & 0 \\
                                                \end{array}
                                              \right)\left(
                                                       \begin{array}{cc}
                                                         \rho(x) & 0 \\
                                                         0 & \rho^{\prime}(x) \\
                                                       \end{array}
                                                     \right)\left(
                                                              \begin{array}{cc}
                                                                0 & 0 \\
                                                                0 & I \\
                                                              \end{array}
                                                            \right)=0;\]
so
$$T(x\otimes 1) = T(x\otimes 1)(I_H\otimes Q_0) = (x\otimes 1_{\alpha})T, \ \ x\in \cl K(H).$$
It follows that $T=I_H\otimes \Lambda$, for some $\Lambda \in
\mathcal{B}(\ell^2,\ell^2_{\alpha})$. We thus have
\[\Psi(x)=A^*(x\otimes 1)A=A^*\rho(x)A=A^*S^*(x\otimes
1_{\alpha})SA=A^*T^*(x\otimes 1_{\alpha})TA.\] Moreover,
\begin{eqnarray*}H^{\alpha}&=&S\mathcal{E}=\overline{S(\mathcal{K}(H)\otimes 1)
AH}=\overline{S(\mathcal{K}(H)\otimes
1)|_{\mathcal{E}}AH}\\&=&\overline{(\mathcal{K}(H)\otimes
1_{\alpha})SAH}=\overline{(\mathcal{K}(H)\otimes
1_{\alpha})TAH}\end{eqnarray*} and thus the representation of $\Psi$
implemented by $TA$ is minimal.
\end{proof}

\begin{remark}\label{T^*T}
{\rm As part of the proof of Lemma \ref{l_minimalscalar}, it was
shown that $T^*T=I_H\otimes Q_0$.  This fact will be used in the
sequel.}
\end{remark}

The main result of this section is the following.

\begin{theorem}\label{p_compl2}
Let $H_2$ be a separable Hilbert space, $\mathcal{D}_2\subseteq
\mathcal{B}(H_2)$ be a unital C*-subalgebra, $H_1\subseteq H_2$ be a
closed subspace such that the projection $p$ from $H_2$ onto $H_1$
belongs to $\mathcal{D}_2$, and $\mathcal{D}_1 = p\mathcal{D}_2p.$
Let $\Phi : \mathcal{K}(H_2)\rightarrow \mathcal{B}(H_2)$ be a
completely positive $\mathcal{D}_2$-bimodule map, and let $\Psi :
\mathcal{K}(H_1)\rightarrow \mathcal{B}(H_1)$ be the map given by
$\Psi(x)=\Phi(x)|_{H_1}, x\in\mathcal{K}(H_1)$. Suppose that the
operator $V\in M_{\beta,1}(\mathcal{D}'_1)$ implements a minimal
representation of $\Psi$. Then there exist an at most countable
cardinal $\gamma \geq \beta$ and an operator $W\in
M_{\gamma,1}(\mathcal{D}_2')$, which implements a minimal representation of $\Phi$, such that $W|_{H_1} = V$.
\end{theorem}
We note that throughout the statement and the proof of the theorem,
we identify $H_1^{\beta}$ with $H_1^{\beta}\oplus 0\subseteq
H_1^{\gamma}$; under this assumption, it will be shown that
$W|_{H_1}$ has range in $H_1^{\beta}$ and is equal to $V$.
\begin{proof}
Let $\Phi(x) = A^*(x\otimes 1)A$, $x\in \mathcal{K}(H_2)$, be any
representation of $\Phi$, where $A\in M_{\infty,1}(\mathcal{D}_2')$.
Then $\Psi(x) = A^*(x\otimes 1)A|_{H_1}$, $x\in \mathcal{K}(H_1)$.
By Lemma \ref{l_minimalscalar}, there exists a minimal
representation
\[\Psi(x)=A^*T^*(x\otimes 1_{\alpha})TA|_{H_1},\ \
x\in\mathcal{K}(H_1),\] where $T=I_{H_1}\otimes
   T_0=(\lambda_{ij}I_{H_1})$
   for
   some
   $T_0=(\lambda_{ij})\in\mathcal{B}(\ell^2,\ell^2_{\alpha})$. Since the representation \mbox{$\Psi(x)=V^*(x\otimes
1_{\beta})V$} is also minimal, there exists \cite{paulsen} a unitary
operator $U : H_1^{\alpha}\rightarrow H_1^{\beta}$ such that
$UTA|_{H_1}=V$ and $U(x\otimes 1_{\alpha})=(x\otimes 1_{\beta})U$,
$x\in \cl K(H_1)$. Hence $U=I_{H_1}\otimes U_0$, where $U_0$ is a
unitary operator in $\mathcal{B}(\ell^2_{\alpha},\ell^2_{\beta} )$,
and thus $\alpha = \beta$. Let $\tilde{U}=I_{H_2}\otimes U_0,\
\tilde{T}=I_{H_2}\otimes T_0$ and
$B=\tilde{U}\tilde{T}A\in\mathcal{B}(H_2,H_2^{\alpha})$. Then
\[B|_{H_1}=\tilde{U}\tilde{T}A|_{H_1}=\tilde{U}\tilde{T}(p\otimes 1)A|_{H_1}=UTA|_{H_1}=V.\]

Let $\Phi^{\prime}:\mathcal{K}(H_2)\rightarrow\mathcal{B}(H_2)$ be
given by
\begin{eqnarray}\label{Phiprime2}\Phi^{\prime}(x)=B^*(x\otimes
1_{\alpha})B, \ \ \  x\in\mathcal{K}(H_2),\end{eqnarray} and let
$\mathcal{E}=\overline{(\mathcal{K}(H_1)\otimes 1)AH_1}$. As in
Lemma \ref{equiv}, the projection $Q$ from $H_1^{\infty}$ onto
$\mathcal{E}$ has the form $I_{H_1}\otimes Q_0$, where $Q_0$ is a
projection in $\mathcal{B}(\ell_2)$. Since $Ap=(p\otimes 1)Ap$, it
follows that \[(I_{H_2}\otimes Q_0^{\perp})Ap=(p\otimes
Q_0^{\perp})Ap=(I_{H_1}\otimes Q_0^{\perp})Ap=0,\] where the last
equality follows from the fact that $ApH_2\subseteq\mathcal{E}$. We
thus have
\begin{eqnarray}\label{perp2}(I_{H_2}\otimes Q_0^{\perp})A=(I_{H_2}\otimes
Q_0^{\perp})Ap^{\perp}.\end{eqnarray} Define
$\Phi^{\prime\prime}:\mathcal{K}(H_2)\rightarrow\mathcal{B}(H_2)$ by
 \begin{eqnarray}\label{Phi2prime2}\Phi^{\prime\prime}(x)=A^*(I_{H_2}\otimes
Q_0^{\perp})(x\otimes 1)(I_{H_2}\otimes Q_0^{\perp})A, \ \
x\in\mathcal{K}(H_2).\end{eqnarray} Then for every $x$ in
$\mathcal{K}(H_2)$, we have that
\begin{eqnarray*}\label{Phis2}
& & \Phi^{\prime}(x)+\Phi^{\prime\prime}(x)\\
& = & A^*\tilde{T}^*\tilde{U}^*(x\otimes
1_{\alpha})\tilde{U}\tilde{T}A+A^*(I_{H_2}\otimes
Q_0^{\perp})(x\otimes 1)(I_{H_2}\otimes
Q_0^{\perp})A\\&=&A^*(x\otimes
1)\tilde{T}^*\tilde{U}^*\tilde{U}\tilde{T}A+A^*(x\otimes
1)(I_{H_2}\otimes Q_0^{\perp})A\\&=&A^*(x\otimes 1)(I_{H_2}\otimes
Q_0)A+A^*(x\otimes 1)(I_{H_2}\otimes Q_0^{\perp})A\\
& = & A^*(x\otimes 1)A = \Phi(x),
\end{eqnarray*}
where the third equality follows from Remark \ref{T^*T} and the fact
that $\tilde{U}$ is unitary.

Now let $\Phi^{\prime\prime\prime}$ be the restriction of
$\Phi^{\prime\prime}$ to $\mathcal{K}(H_2\ominus H_1)$ so that, for
$x\in\mathcal{K}(H_2\ominus H_1)$, we have
\begin{eqnarray*}
\Phi^{\prime\prime\prime}(x) & = & A^*(I_{H_2}\otimes
Q_0^{\perp})(x\otimes 1)(I_{H_2}\otimes Q_0^{\perp})A|_{H_2\ominus
H_1}\\ & = & A^*(I_{H_2}\otimes Q_0^{\perp})(x\otimes
1)(I_{H_2}\otimes Q_0^{\perp})A.
\end{eqnarray*}
Using Lemma \ref{l_minimalscalar} , find a minimal representation
\[\Phi^{\prime\prime\prime}(x)=A^*(I_{H_2}\otimes
Q_0^{\perp})R^*(x\otimes 1_{\delta})R(I_{H_2}\otimes Q_0^{\perp})A\]
where $R$ is a bounded operator from $(H_2\ominus H_1)^{\infty}$
into $(H_2\ominus H_1)^{\delta}$ that can be expressed as
$R=I_{H_2\ominus H_1}\otimes R_0$ for some $R_0$ in
$\mathcal{B}(\ell^2,\ell^2_{\delta})$, $\delta$ being an at most
countable cardinal. By construction, $R(x\otimes 1)R^*=x\otimes
1_{\delta}$.  Letting \linebreak $C=R(I_{H_2}\otimes Q_0^{\perp})A$,
equation (\ref{perp2}) implies that
\begin{eqnarray}\label{Cperp2}Cp^{\perp}=C\end{eqnarray} and hence we have that for each $x\in\mathcal{K}(H_2)$,
\begin{eqnarray*}C^*(x\otimes 1_{\delta})C&=&p^{\perp}C^*(x\otimes
1_{\delta})Cp^{\perp}=C^*(p^{\perp}xp^{\perp}\otimes
1_{\delta})C\\&=&\Phi^{\prime\prime}(p^{\perp}xp^{\perp})=A^*(I_{H_2}\otimes
Q_0^{\perp})(p^{\perp}xp^{\perp}\otimes 1)(I_{H_2}\otimes
Q_0^{\perp})A\\&=&p^{\perp}A^*(I_{H_2}\otimes Q_0^{\perp})(x\otimes
1)(I_{H_2}\otimes Q_0^{\perp})Ap^{\perp}\\
& \stackrel{\text{by } (\ref{perp2})}{=} & A^*(I_{H_2}\otimes
Q_0^{\perp})(x\otimes 1)(I_{H_2}\otimes Q_0^{\perp})A =
\Phi^{\prime\prime}(x).\end{eqnarray*}

The operators $B:H_2\rightarrow H_2^{\alpha}$ and $C:H_2\rightarrow
(H_2\ominus H_1)^{\delta}$ can be expressed as columns of length
$\alpha$ and $\delta$, respectively; say, $B=(b_1,b_2,\dots) ^t$ and
$C=(c_1,c_2,\dots)^t$, where $b_i, c_j\in\mathcal{D}_2'$ for all $i,
j$. Let $U$ be the column operator with entries $B$ and $C$, that
is,
\begin{eqnarray}\label{U2}
U=(B^t, C^t)^t.\end{eqnarray} Suppose that
$\sum_{i=1}^{\alpha}\lambda_i b_i + \sum_{j=1}^{\delta}\mu_{j}c_j=0$
for some $(\lambda_i) _{i=1}^{\alpha}\in \ell_{\alpha}^2$,
\mbox{$(\mu_j)_{j=1}^{\delta}\in \ell^2_{\delta}$.}  Then
$\sum_{i=1}^{\alpha}\lambda_i pb_i +
\sum_{j=1}^{\delta}\mu_{j}pc_j=0$. Since $Cp=0$, we have that
$\sum_{i=1}^{\alpha}\lambda_i pb_i=0$ and since $B|_{H_1}=V$
implements a minimal representation, we have by Lemma \ref{equiv}
that the entries of $Bp$ are strongly independent and hence
$\lambda_i=0$ for all $i$. Consequently,
$\sum_{j=1}^{\delta}\mu_{j}c_j=0$ and the minimality of the
representation associated with $C$ implies, again by Lemma
\ref{equiv}, that $\mu_j=0$ for all $j$.

Let
\[b_i'=\left\{\begin{array}{ll}b_i,&1\leq
i\leq\alpha,\\0,&i>\alpha,\end{array}\right.\ \ \
c_i'=\left\{\begin{array}{ll}c_i, & 1\leq
i\leq\delta,\\0,&i>\delta,\end{array}\right.\] and set
$W=(b_1',c_1',b_2',c_2',\dots)^t$ --- note that in the case in which
both cardinals are finite the sequence has finitely many non-zero
terms. In the case where both $\alpha$ and $\delta$ are infinite,
the series $\sum_{i=1}^{\infty}b_i^*xb_i+c_i^*xc_i$ is easily seen
to converge weak* to
$\sum_{i=1}^{\infty}b_i^*xb_i+\sum_{i=1}^{\infty} c_i^*xc_i$. It now
follows that $\Phi(x) = W^*(x\otimes 1_{\alpha+\delta})W$, $x\in \cl
K(H_2)$. By Lemma \ref{equiv}, the representation of $\Phi$
implemented by $W$ is minimal. We note that
\[W|_{H_1}=(b_1',c_2',b_2',c_2',\dots)^t|_{H_1}=(b_1,0,b_2,0,\dots)|_{H_1}=B|_{H_1}=V,\]
where we used (\ref{Cperp2}) to obtain the second equality and the
third equality follows as a result of the identification made at the
start of the proof.
\end{proof}

Next we show how Theorem~\ref{p_compl2} can be
applied to obtain a result closely related to Theorem \ref{c_c2}.

\begin{theorem}\label{repv}
Let $\mathcal{D}\subseteq \mathcal{B}(H)$ be a unital C*-subalgebra,
$(p_n)_{n=1}^{\infty}\subseteq\mathcal{D}$ be an increasing sequence
of projections such that $\vee_{n\in \mathbb{N}} p_n=I$, $H_n=p_nH$
and $\mathcal{D}_n=p_n\mathcal{D}p_n$. Let
$\Phi:\cup_{n=1}^{\infty}\mathcal{K}(H_n)\rightarrow\cup_{n=1}^{\infty}\mathcal{B}(H_n)$
be a map such that $\Phi|_{\mathcal{K}(H_n)}$ is completely positive
and $\mathcal{D}_n$-bimodular. Then there exists a family
$(a_i)_{i=1}^{\infty}$ of closable operators affiliated with
$\mathcal{D}'$ such that
$$\Phi(x)=\sum_{i=1}^\infty a_i^*xa_i, \ \ x\in
\cup_{n=1}^{\infty}\mathcal{K}(H_n).$$
\end{theorem}

In the course of the proof we will encounter operators
$V_n:H_n\rightarrow(H_n)^{\alpha_n}$ and
$V_{n+1}:H_{n+1}\rightarrow(H_{n+1})^{\alpha_{n+1}}$ for cardinals
$\alpha_n\leq\alpha_{n+1}$; we will identify throughout the proof
$(H_n)^{\alpha_n}$ with the subspace $(H_n)^{\alpha_n}\oplus 0$ of
$(H_n)^{\alpha_{n+1}}$.

\begin{proof}
Let $H_0 = \cup_{n=1}^{\infty}H_n$ and
$\Phi_n:\mathcal{K}(H_n)\rightarrow\mathcal{B}(H_n)$ be the map
given by $\Phi_n(x)=\Phi(x)|_{H_n}$ for $x\in\mathcal{K}(H_n)$.
Using analogous notation and repeating the process used to obtain
the operator $U$ in identity (\ref{U2}), we can form an operator
$U_n : H_n\rightarrow\oplus_{i=1}^n H_n^{\alpha_i-\alpha_{i-1}}$
such that (if we allow a slight abuse of notation and let $x\otimes
1_{\alpha_n}$ also denote $(x\otimes
1_{\alpha_1})\otimes\dots\otimes(x\otimes
1_{\alpha_n-\alpha_{n-1}})$) we have $\Phi_n(x)=U_n^*(x\otimes
1_{\alpha_n})U_n, x\in\mathcal{K}(H_n)$, and $U_n|_{H_m}=U_m$ for
all $m\leq n$.  By construction, $U_n$ is equal to

\begin{eqnarray*}\left(\begin{array}{cc}
\begin{array}{c}
\begin{array}{cc}
\left(\begin{array}{c}b_{1,1}^{(1)}\\b_{1,2}^{(1)} \\ \vdots
\end{array} \right)+ &
 \left(\begin{array}{c}b_{1,1}^{(2)}-b_{1,1}^{(1)}\\b_{1,2}^{(2)}-b_{1,2}^{(1)} \\
\vdots\end{array}\right)   \\
 &  \left(\begin{array}{c}c_{2,1}^{(2)}\\c_{2,2}^{(2)}\\
\vdots\end{array}\right)\end{array}+\dots + \\
\vspace{30mm}\end{array} &
\left(\begin{array}{c}b_{1,1}^{(n)}-b_{1,1}^{(n-1)}\\b_{1,2}^{(n)}-b_{1,2}^{(n-1)}\\
\vdots\vspace{3mm}\\
b_{2,1}^{(n)}-b_{2,1}^{(n-1)}\\b_{2,2}^{(n)}-b_{2,2}^{(n-1)}\\
\vdots\\ \vdots\\ \vdots\\ \vdots\\ b_{n-1,1}^{(n)}-c_{n-1,1}^{(n)}\\b_{n-1,2}^{(n)}-c_{n-1,2}^{(n)}\\ \vdots\end{array}\right)\\
  & \left(\begin{array}{c}c_{n,1}^{(n)}\\c_{n,2}^{(n)}\\
\vdots\end{array}\right)
\end{array}\right)
=
\end{eqnarray*}
$$\left(b_{1,1}^{(n)}\\b_{1,2}^{(n)}
\cdots b_{2,1}^{(n)}\\b_{2,2}^{(n)}\cdots \cdots
b_{n,1}^{(n)}\\b_{n,2}^{(n)}\cdots\right)^t.$$
where for convenience of notation we assume the column operators
have infinitely many terms; it should be noted however that these
may in fact terminate at some finite value. We note that
$b_{r,s}^{(t)}\in\mathcal{D}'p_t$. Since
$b_{r,s}^{(n)}|_{H_{n-1}}=b_{r,s}^{(n-1)}$, we can form a densely
defined operator $b_{r,s}$ by letting
$b_{r,s}|_{H_n}=b_{r,s}^{(n)}$. Standard arguments show that
$b_{r,s}$ is closable and affiliated with $\mathcal{D}'$. Define an
operator $V:H_0\rightarrow (H_0)^{\infty}$ by
\[V=(a_1,a_2,\dots)^t\overset{def}=\left(b_{1,1},b_{1,2},b_{2,1},b_{1,3},b_{2,2},b_{3,1},b_{1,4},\dots\right)^t,\]
where we have, if necessary, extended $V$ to an infinite column
operator.  From this we can obtain an operator
$V_n:H_n\rightarrow(H_n)^{\alpha_n}$ by observing that when
restricted to $H_n$, all terms $b_{i,j}$ for which $i>n$ vanish, and
each of the remaining $\alpha_n$ terms is such that
$b_{i,j}|_{H_n}=b_{i,j}^n$. This also implies that $V_n|_{H_m}=V_m$
for all $m\leq n$ (in the sense described before the start of the
proof). It can easily be seen that
\[V_n^*(x\otimes
1_{\alpha_n})V_n=U_n^*(x\otimes 1_{\alpha_n})U_n.\] and that $V_n$
is a bounded column operator.

Finally, given $x\in\cup_{n=1}^{\infty}\mathcal{K}(H_n)$, fix $n$
such that $x\in\mathcal{K}(H_n)$ and notice that
\[\Phi(x)=\Phi_n(x)=U_n^*(x\otimes 1_{\alpha_n})U_n=V_n^*(x\otimes 1_{\alpha_n})V_n=V^*(x\otimes
1)V=\sum_{i=1}^{\infty}a_i^*xa_i.\] We note that, given
$x\in\mathcal{K}(H_n)$, we have that
\[\sum_{i=1}^{\infty}a_i^*xa_i=\sum_{i=1}^{\infty}a_i^*p_nxp_na_i=\sum_{i=1}^{\infty}(p_na_i^*p_n)x(p_na_ip_n).\]
Since $p_na_ip_n$ is a bounded operator for all $k\in\mathbb{N}$ and
the sequence $(p_na_ip_n)_{k=1}^{\infty}$ defines the bounded column
operator $V_n$, this series indeed converges in the weak$^*$
topology.
\end{proof}

\begin{remark}\label{r_rc}{\rm In the course of the proof of Theorem \ref{repv} it
was shown that $b_{r,s}p_n=b_{r,s}^{(n)}\in\mathcal{D}'p_n$ for all
$r,s\in\mathbb{N}, n\geq r$; thus, $a_ip_n\in\mathcal{D}'p_n$ for
all $i,n\in\mathbb{N}$.  This will used in the sequel.}\end{remark}

\section{Positive local Schur multipliers}\label{s_plsm}

In this section we examine positive local Schur multipliers. The
main results of the section are the characterisation contained in
Theorem \ref{th_lps}, the lifting established in Corollary
\ref{c_mfa} and the characterisation result of Theorem \ref{th_toe}.

We begin by recalling some definitions, known results and notational
conventions. We assume throughout the section that $(X,\mu)$ and
$(Y,\nu)$ are standard ($\sigma$-finite) measure spaces, equip
$X\times Y$ with the product measure and denote by $\mathfrak{B}(X)$
the linear space of all measurable functions on $X$. If $f\in
L^{\infty}(X)$, we write $M_f$ for the operator of multiplication by
$f$ acting on $L^2(X)$, and let $\cl D = \{M_f : f\in
L^{\infty}(X)\}$. The characteristic function of a measurable subset
$\alpha\subseteq X$ will be denoted by $\chi_{\alpha}$. For each
$k\in L^2(X\times Y)$, we let $T_k \in \mathcal{C}_2(L^2(X),L^2(Y))$
be the operator given by
\[(T_k\xi)(y)=\int_Xk(x,y)\xi(x)d\mu(x), \ \ \xi\in L^2(X)\]
and $S_{\varphi} : \mathcal{C}_2(L^2(X),L^2(Y))\rightarrow
\mathcal{C}_2(L^2(X),L^2(Y))$ be the map given by $S_{\varphi}(T_k)$
$=$ $T_{\varphi k}$, $k\in L^2(X\times Y)$. Recall that the map
$S_{\varphi}$ is called a \emph{(measurable) Schur multiplier} if
$S_{\varphi}$ is bounded in $\|\cdot\|_{\op}$.

We next recall some notions from \cite{a} and \cite{stt2011}. Two
subsets $E, F\subseteq X$ are called \emph{equivalent} (written
$E\sim F$) if their symmetric difference is contained in a null set.
A subset of $X\times Y$ is said to be a \emph{rectangle} if it has
the form $\alpha\times\beta$, where $\alpha\subseteq X$ and
$\beta\subseteq Y$ are measurable.  A subset $E\subseteq X\times Y$
is called \emph{marginally null} if $E\subseteq (X_0\times Y)\cup
(X\times Y_0)$, where $\mu(X_0)=\nu(Y_0)=0$.  We call two subsets
$E,F\subseteq X\times Y$ \emph{marginally equivalent} (and write
\mbox{$E\simeq F$)} if their symmetric difference is marginally
null. A measurable function \mbox{$\nph : X\times Y\rightarrow
\bb{C}$} is called \emph{$\omega$-continuous} if $\nph^{-1}(U)$ is
marginally equivalent to a countable union of rectangles, for every
open subset $U\subseteq \bb{C}$. A countable family of rectangles is
called a \emph{covering family for $X\times Y$} if its union is
marginally equivalent to $X\times Y$. We say that a function
$\varphi\in\mathfrak{B}(X\times Y)$ is a \emph{local Schur
multiplier} if there exists a covering family
$\{\kappa_m\}_{m=1}^{\infty}$ of rectangles in $X\times Y$ such that
$\varphi|_{\kappa_m}$ is a Schur multiplier, for all $m$.

\begin{proposition}\label{p_charin}
For a function $\varphi\in\mathfrak{B}(X\times X)$, the following
are equivalent:
\begin{enumerate}[(i)]\item $\nph$ is a local Schur multiplier;
\item there exists an increasing sequence $(X_n)_{n=1}^{\infty}$ of measurable subsets of $X$ such that
$X\setminus(\cup_{n=1}^{\infty}X_n)$ is null and
$\varphi|_{X_n\times X_n}$ is a Schur multiplier for each $n$.
\end{enumerate}
\end{proposition}
\begin{proof}
(i)$\Rightarrow$(ii)
Suppose that $\varphi$ is a local Schur multiplier and
let $\{\kappa_m\}$ be a covering family for $X\times X$ such that
$\varphi|_{\kappa_m}$ is a Schur multiplier. By \cite[Lemma
3.4]{stt2011}, there exists a pairwise disjoint family
$\{Y_i\}_{i=1}^{\infty}\subseteq X$ such that
$\cup_{i=1}^{\infty}Y_i$ is equivalent to $X$ and each $Y_i\times
Y_j$ is contained in a finite union of sets from
$\{\kappa_m\}_{m=1}^{\infty}$. By  \cite[Lemma 2.4 (ii)]{stt2011},
$\varphi|_{Y_i\times Y_j}$ is a Schur multiplier. Let
$X_n=\cup_{i=1}^nY_i$; then
$(X_n)_{n\in\mathbb{N}}$ is an increasing sequence, $X_n\times X_n =
\cup_{i,j=1}^nY_i\times Y_j$ and
$\cup_{n=1}^{\infty}X_n\sim\cup_{i=1}^{\infty}Y_i\sim X$. By
\cite[Lemma 2.4 (ii)]{stt2011}, $\varphi|_{X_n\times X_n}$ is a
Schur multiplier and the proof is complete.

(ii)$\Rightarrow$(i) The collection $(X_n\times X_n)_{n=1}^{\infty}$
is a covering family and thus $\varphi$ is a local Schur multiplier.
\end{proof}

We will denote by $\mathcal{C}_2(H)^+$ the cone of all positive
operators in $\mathcal{C}_2(H)$. In view of Proposition
\ref{p_charin}, it is natural to introduce the following notions.

\begin{definition}\label{d_iii}
Let $(X,\mu)$ be a standard measure space and $\varphi$ be a
measurable function on $X\times X$. We say that $\varphi$ is a

(i) \emph{positive Schur multiplier} if the map $S_{\varphi}$ is
bounded in $\|\cdot\|_{\op}$ and leaves $\mathcal{C}_2(L^2(X))^+$
invariant;

(ii) \emph{positive local Schur multiplier} if there exists an
increasing sequence $(X_n)_{n=1}^{\infty}$ of measurable subsets of
$X$ such that $X\setminus (\cup_{n=1}^{\infty}X_n)$ is null and
$\varphi|_{X_n\times X_n}$ is a positive Schur multiplier for every
$n$.
\end{definition}

It is immediate that every positive Schur multiplier is a Schur
multiplier and that every positive local Schur multiplier is a local
Schur multiplier.

R. R. Smith has established an automatic complete boundedness result
for maps, modular over C*-algebras with a cyclic vector
\cite[Theorem 2.1]{smith1991}. We will need the following automatic
complete positivity result; we omit its proof since it follows
closely the ideas in Smith's proof.

\begin{lemma}\label{pcp}
Let $H$ be a Hilbert space, $\cl E\subseteq \cl B(H)$ be an operator
system, and $\cl B\subseteq \cl B(H)$ be a C*-algebra with a cyclic
vector such that $\cl B \cl E \cl B \subseteq \cl E$.
Then every positive $\cl B$-bimodule map $\Phi :
\mathcal{E}\rightarrow\mathcal{B}(H)$ is completely positive.
\end{lemma}

We can now formulate and prove one of the main results of this section.

\begin{theorem}\label{th_lps}
A function $\varphi\in\mathfrak{B}(X\times X)$ is a positive local
Schur multiplier if and only if there exists a measurable function
$a : X\rightarrow \ell^2$ such that $\varphi(x_1,x_2) =
(a(x_1),a(x_2))_{\ell^2}$ almost everywhere on $X\times X$.
\end{theorem}
\begin{proof}
We let $H_0 = \cup_{n=1}^{\infty}L^2(X_n)$.  Suppose that $\varphi$
is a positive local Schur multiplier and let $(X_n)_{n=1}^{\infty}$
be the sequence of subsets of $X$ from Definition \ref{d_iii} (ii).
We can, moreover, assume that $\mu(X_n) < \infty$. Recall that
$\mathcal{D} = \{M_f : f\in L^{\infty}(X)\}$. The projection $p_n$
from $L^2(X)$ onto $L^2(X_n), n\in\mathbb{N},$ is then given by $p_n
= M_{\chi_{X_n}}$. We identify $\mathcal{C}_2(L^2(X_n))$ with a
subspace of $\mathcal{C}_2(L^2(X_{n+1}))$ in the natural way. Let
$$S_{\varphi}:\cup_{n=1}^{\infty}\mathcal{C}_2\left(L^2(X_n)\right)\rightarrow\cup_{n=1}^{\infty}\mathcal{C}_2\left(L^2(X_n)\right)$$
be the map given by $S_{\varphi}(T_k) = T_{\chi_{X_n\times
X_n}\varphi k}$,  $k\in L^2(X_n\times X_n)$. We have that the
restriction $S_{\varphi}|_{\mathcal{C}_2(L^2(X_n))}$ is positive,
bounded and $\mathcal{D}_n$-bimodular. Hence $S_{\nph}$ satisfies
the conditions of Theorem \ref{c_c2} (or those of Theorem
\ref{repv}), and thus there exists a linear operator $V :
H_0\rightarrow(H_0)^{\infty}$ of the form $V = (M_{a_1},
M_{a_2},\dots)^t$, where $a_i\in\mathfrak{B}(X)$, $i\in\mathbb{N}$,
such that
\[S_{\varphi}(T_k)=V^*(T_k\otimes
1)V=\sum_{i=1}^{\infty}M_{a_i}^*T_kM_{a_i} \text{ for all
}T_k\in\cup_{n=1}^{\infty} \cl C_2\left(L^2(X_n)\right).\]

Fix $n\in\mathbb{N}$. We have that $\esssup_{x\in
X_n}\sum_{i=1}^{\infty}|a_i(x)|^2=\|\sum_{i=1}^{\infty}M_{|a_i\chi_{X_n}|^2}\|=\|V
p_n\|^2$. It follows that $\sum_{i=1}^{\infty}|a_i(x)|^2<\infty$ for
almost all $x\in X$. Thus the function $a : X\rightarrow \ell^2$
given by $a(x) = (a_i(x))_{i=1}^{\infty}$, $x\in X$, is well-defined
up to a null set.

Let $\psi=\sum_{i=1}^{\infty}a_i\otimes \overline{a_i}$. Then
$$T_{\nph k}=S_\nph (T_k) = \sum_{i=1}^{\infty}M_{a_i}^*T_kM_{a_i}=T_{\psi k}, \ \ k\in
L^2(X_n\times X_n), \ n\in\mathbb{N},$$ This implies $\nph=\psi$
almost everywhere on $X_n\times X_n$ and consequently
$$\varphi(x_1,x_2)=\sum_{i=1}^{\infty}a_i(x_1)\overline{a_i(x_2)} = (a(x_1),a(x_2))_{\ell^2},$$
for almost all $(x_1,x_2)\in \cup_{n=1}^{\infty}(X_n\times X_n)$,
and hence for almost all $(x_1,x_2)\in X\times X$.

Conversely, suppose that there exists a function $a : X\rightarrow
\ell^2$, say $a(x) = (a_i(x))_{i\in \bb{N}}$, $x\in X$, such that
$\varphi(x_1,x_2) = (a(x_1),a(x_2))_{\ell^2}$ almost everywhere on
$X\times X$. Let $X_n=\{x\in X : \|a(x)\|^2\leq n\}$ and observe
that $\cup_{i=1}^{\infty}X_n\sim X$. For $k\in L^2(X_n\times X_n)$,
we have that
\mbox{$S_{\varphi}(T_k)=\sum_{i=1}^{\infty}M_{a_i\chi_{X_n}}^*T_kM_{a_i\chi_{X_n}}$}
and hence $S_{\varphi}|_{\mathcal{C}_2(L^2(X_n))}$ is a bounded
positive map. Consequently, $\varphi$ is a positive local Schur
multiplier.
\end{proof}

\noindent{\bf Remark.} Let $(X,\mu)$ be a standard measure space and $\mu'$ be a measure defined on the same
$\sigma$-algebra and absolutely continuous with respect to $\mu$.
If $\varphi$ is a positive local Schur multiplier with respect to $\mu$ then it is so with respect to $\mu'$.
Indeed, this follows immediately from the representation  given in Theorem~\ref{th_lps}.

\medskip

As a corollary of Theorem \ref{th_lps}, we also note the following,
rather well-known, description of positive Schur multipliers.

\begin{corollary}\label{psmc}
A function $\varphi\in L^{\infty}(X\times X)$ is a positive Schur
multiplier if and only if there exists a measurable function $a :
X\rightarrow \ell^2$ such that $\esssup_{x\in X}\|a(x)\| < \infty$
and $\nph(x_1,x_2) = (a(x_1),a(x_2))_{\ell^2}$ almost everywhere on
$X\times X$.
\end{corollary}

Let $\varphi\in L^{\infty}(X\times X)$ be a positive Schur
multiplier. Then there are potentially many functions $a :
X\rightarrow \ell^2$ satisfying the conclusion of Corollary
\ref{psmc}; we call them \emph{representing functions for $\nph$}.
For each such function, say, $a(x) = (a_i(x))_{i\in \bb{N}}$ ($x\in
X$), we have a corresponding Stinespring representation of the
completely positive map $S_{\nph}$:
\begin{equation}\label{eq_snph}
S_{\nph}(T) = \sum_{i=1}^{\infty} M_{a_i}^*TM_{a_i}, \ \ \ T\in \cl
K(L^2(X)).
\end{equation}
Let us call $a$ a \emph{minimal representing function for $\nph$} if
the representation (\ref{eq_snph}) of $S_{\nph}$ is minimal.

\begin{proposition}\label{p_minsh}
Let $\nph\in L^{\infty}(X\times X)$ be a positive Schur multiplier
and $a : X\rightarrow \ell^2$ be a representing function for $\nph$.
The following are equivalent:

(i) \ $a$ is minimal;

(ii) for each null set $M\subseteq X$, the set $\{a(x) : x\in
X\setminus M\}$ has dense linear span in $\ell^2$.
\end{proposition}
\begin{proof}
Suppose that $a_i$, $i\in \bb{N}$, are the coordinate functions of
$a$. By Lemma \ref{equiv}, $a$ is not a minimal representing
function for $\nph$ if and only if $\{M_{a_i}\}_{i\in \bb{N}}$ is
not strongly independent, if and only if there exists $0 \neq
(\lambda_i)_{i\in \bb{N}} \in \ell^2$ such that $\sum_{i=1}^{\infty}
\lambda_i M_{a_i} = 0$, if and only if there exists a null set
$M\subseteq X$ such that $\sum_{i=1}^{\infty} \lambda_i a_i(x) = 0$ for
all $x\in X\setminus M$, if and only if there exists a null set
$M\subseteq X$ such that the linear span of $\{a(x) : x\in
X\setminus M\}$ is not dense in $\ell^2$.
\end{proof}

The following lifting result for positive Schur multipliers now
follows from Theorem \ref{p_compl2}.

\begin{corollary}\label{c_mfa}
Let $\nph \in L^{\infty}(X\times X)$ be a positive Schur multiplier,
and let $Y\subseteq X$ be a measurable subset. Suppose that $a :
Y\rightarrow\ell^2$ is a minimal representing function for
$\nph|_{Y\times Y}$. Then there exists a minimal representing
function $b : X\rightarrow \ell^2 \oplus \ell^2$ for $\nph$ such
that $b(x) = a(x) \oplus 0$ for all $x\in Y$.
\end{corollary}

Let $X$ be a set and $\varphi : X\times X \rightarrow
\bb{C}$ be a function. We recall that $\nph$ is called \emph{positive definite}
if $(\nph(x_i,x_j))_{i,j = 1}^N$ is a positive
matrix for all $x_1,x_2,\dots,x_N \in X$ and all $N\in \bb{N}$.
In the proof of the following proposition, we will use the
following well-known fact:
if $X$ is a locally compact
Hausdorff space, $\mu$ is a regular Borel measure on $X$, $K\subseteq X$ is compact and $k\in
L^2(X\times X)$ is continuous and positive definite on $K\times K$ then the (Hilbert-Schmidt)
operator $T_k$ is positive on $L^2(X,\mu)$.

\medskip

The motivation behind part (ii) of the next theorem is \cite[Theorem 9.3]{ks2006},
where a relation between $\omega$-continuous measurable Schur multipliers and
classical Schur multipliers is established.

\begin{theorem}\label{c_pd}
Let $\varphi\in\mathfrak{B}(X\times X)$.

(i) \ $\varphi$ is a positive local
Schur multiplier if and only if $\varphi$ is a local Schur
multiplier and $\varphi$ is equivalent to a positive definite
function.

(ii) Suppose that $\varphi$ is $\omega$-continuous.
Then $\varphi$ is a positive local Schur multiplier if and only if there exist a null set $X_0$
and an increasing sequence $\{Y_k\}$ of measurable subsets of $X$ such that $X\setminus X_0=\cup_{k=1}^\infty Y_k$
and $\varphi|_{Y_k\times Y_k}$ is a positive Schur multiplier with respect to the counting measure on $Y_k$, for every $k$.
\end{theorem}
\begin{proof}
(i) Suppose that $\varphi$ is a positive local Schur multiplier.  By
Theorem \ref{th_lps}, there exists a measurable function $a :
X\rightarrow \ell^2$ such that $\varphi(x_1,x_2) =
(a(x_1),a(x_2))_{\ell^2}$ almost everywhere on $X\times X$. On the
other hand, a straightforward verification shows that the function
$(x_1,x_2)\rightarrow (a(x_1),a(x_2))_{\ell^2}$ is positive
definite.

Conversely, suppose that $\varphi$ is a local Schur multiplier and a
positive definite function. We may assume that $X$ is a
$\sigma$-compact metric space and $\mu$ is a regular Borel measure.
Let $(X_n)_{n\in\mathbb{N}}$ be the increasing sequence of
measurable sets arising from Proposition \ref{p_charin}. Let $k =
\sum_{i=1}^m f_i\otimes \overline{f_i}$ for some $f_1,\dots,f_m \in
L^2(X_n)$, so that the operator $T_k$ is a positive and of finite
rank. Fix $\epsilon > 0$. A successive application of Lusin's
Theorem shows that there exists a compact subset $K_{\epsilon}\subseteq X_n$,
whose complement in $X_n$ has measure less than $\epsilon$, such that
$f_i|_{K_{\epsilon}}$ is continuous, for each $i = 1,\dots,m$. On
the other hand, by \cite[Proposition 3.2]{stt2011}, $\nph$ is
$\omega$-continuous and by Lusin's Theorem for $\omega$-continuous
functions \cite[Theorem 8.3]{ks2006}, there exists a compact set
$L_{\epsilon}\subseteq X$, whose complement has measure less than
$\epsilon$, such that $\nph$ is continuous on $L_{\epsilon}\times
L_{\epsilon}$. It follows that $\nph k$ is positive definite on
$(K_{\epsilon}\cap L_{\epsilon})\times (K_{\epsilon}\cap
L_{\epsilon})$, and by the result stated before the statement of the
theorem, $P_{\epsilon}T_{\nph k}P_{\epsilon}\geq 0$, where
$P_{\epsilon}$ is the projection of multiplication by the
characteristic function of $K_{\epsilon}\cap L_{\epsilon}$.
Letting $\epsilon$ tend to zero, we get that $T_{\nph k} \geq 0$.
Since $\nph|_{X_n\times X_n}$ is a Schur multiplier, it follows that
$\nph|_{X_n\times X_n}$ is a positive Schur multiplier.
Thus, $\nph$ is a positive local Schur multiplier.

(ii) Let $a:X\to \ell^2$ be a function defined almost everywhere on
$X$ such that $\varphi(x,y)=(a(x),a(y))_{\ell^2}$ almost everywhere.
Since both functions in the last equality are $\omega$-continuous,
we have that they are equal marginally almost everywhere. Hence
there exists a null set $X_0$ such that $\varphi$ is positive
definite on $(X\setminus X_0)\times (X\setminus X_0)$. Letting
$Y_k=\{x\in X\setminus X_0: \|a(x)\|_{\ell^2}\leq k\}$ we obtain the
result. The converse follows from part (i) and \cite[Theorem 9.3]{ks2006}.
\end{proof}

\begin{example}\rm
Let $\nph(x,y)=1/(x+y)$, $x,y\in{\mathbb R^+}$, $(x,y)\ne (0,0)$.
Then $\nph\in{\mathfrak B}(\mathbb R^+\times\mathbb R^+,
\lambda\times \lambda)$, where $\lambda$ is the Lebesque measure. We
have $$\nph(x,y)=\int_0^{+\infty}e^{-sx}e^{-sy}ds=(e^{-\cdot x},e^{-
\cdot y})_{L^2(\mathbb R^+)}$$ and $\|e^{-sx}\|^2_{L^2(\mathbb
R^+)}=1/(2x)$. Expressing the function $e^{-\cdot x}$ in terms of an
orthonormal basis of $L^2(\mathbb R^+)$, we can find a measurable
function $a:\mathbb R^+\to \ell^2$ such that $\nph(x,y) =
(a(x),a(y))_{\ell^2}$ almost everywhere and $\|a(x)\|_{\ell^2} <
\infty$ for almost all $x\in \mathbb R^+$. Hence $\nph(x,y)$ is a
positive local Schur multiplier. It is not a Schur multiplier since
$\nph\notin  L^{\infty}(\mathbb R^+\times \mathbb R^+)$.
\end{example}

The previous example is taken from the theory of operator monotone
functions. It is known that if a function $f$ is continuously
differentiable on an interval $(a,b)$ then $f$ is operator monotone
on $(a,b)$ if and only if the divided difference $\check{f}$ given
by $\check f(x,y) = (f(x)-f(y))/(x-y)$, $x\neq y$ and $\check f(x,y)
= f'(x)$, $x\in (a,b)$, is positive definite on $(a,b)\times (a,b)$
(see, for example, \cite{hp}, where a proof of this fact is given
using Schur multiplier techniques). Operator monotonicity is related
to positivity of local Schur multipliers in the following way.

\begin{proposition}\label{p_lpsd}
Let $f : (a,b)\rightarrow \bb{C}$ be a continuously differentiable function. The
divided difference $\check{f}$ is a positive local Schur multiplier
on $(a,b)\times (a,b)$ with respect to any choice of a standard Borel measure on $(a,b)$ if and only if $f$ is operator monotone.
\end{proposition}
\begin{proof}
Suppose that
$\check{f}$ is a positive local Schur multiplier on $(a,b)\times (a,b)$ with respect to
every standard Borel measure. Let $F\subseteq (a,b)$ be a finite set and let $\mu_F$ be the measure
given by $\mu_F(\alpha) = |\alpha\cap F|$ for a Borel set $\alpha$. Our assumption implies that
there exists a Borel set $Y\subseteq (a,b)$ with $F\subseteq Y$ such that $\check{f}|_{Y\times Y}$ is a
positive Schur multiplier with respect to $\mu_F$. It follows that $\check{f}|_{F\times F}$ is a positive Schur multiplier
(with respect to $\mu_F$), and hence $\check{f}|_{F\times F}$ is a positive matrix. Since this is true for all finite sets $F\subseteq (a,b)$,
we have that $\check{f}$ is a positive definite function. By
\cite{hp}, $f$ is operator monotone.

Conversely, suppose that $f$ is operator monotone and let $\mu$ be a
standard Borel measure on $(a,b)$; by \cite{hp}, $\check{f}$ is
positive definite. Let $U_n = \{x\in (a,b) : f'(x) < n\}$; then
$\cup_{n\in \bb{N}} U_n = (a,b)$. Let $n\in \bb{N}$ and $F\subseteq
U_n$ be a finite subset. Since $\check{f}|_{F\times F}$  is a
positive matrix, the norm of its corresponding Schur multiplication
is bounded by $\max_{x\in F} f'(x)$, which does not exceed $n$. It
follows that $\check{f}|_{U_n\times U_n}$ is a Schur multiplier with
respect to the counting measure. By \cite[Theorem 9.3]{ks2006},
$\check{f}|_{U_n\times U_n}$ is a Schur multiplier with respect to
$\mu$. Hence, $\check{f}$ is a local Schur multiplier with respect
to $\mu$. Now Theorem \ref{c_pd} shows that $\check{f}$ is a
positive local Schur multiplier with respect to $\mu$.
\end{proof}

\subsection{Positive Multipliers of Toeplitz type}

We conclude this section by considering positive multipliers of
Toeplitz type. Let $G$ be a locally compact group equipped with left
Haar measure and $N$ be the map sending a measurable function
$f:G\rightarrow\mathbb{C}$ to the function $Nf:G\times
G\rightarrow\mathbb{C}$ given by $Nf(s,t) = f(st^{-1})$; we call the
functions of the form $Nf$ \emph{functions of Toeplitz type}. It was shown
in \cite{bf} that if $f\in L^{\infty}(G)$ then $Nf$ is a Schur
multiplier if and only if $f$ is equivalent to an element of
$M^{\cb}A(G)$ (the latter being the set of all completely bounded
multipliers of the Fourier algebra $A(G)$ of $G$). On the other
hand, it was proved in \cite{stt2011} that if $G$ is abelian then
$Nf$ is a local Schur
multiplier if and only if $f$ is equivalent to a function that
belongs locally to $A(G)$ at every point of the group $G$.
In particular, examples were given of local Schur multipliers $\nph$ of
Toeplitz type that are not Schur multipliers. The following
proposition shows that this cannot happen with the additional
assumption that $\nph$ be positive, provided $G$ is amenable.

\begin{theorem}\label{th_toe}
Let $G$ be an amenable locally
compact group, $f : G\rightarrow \bb{C}$ be a measurable function
and $\nph = Nf$. The following are equivalent:

(i) \ $\nph$ is a positive Schur multiplier;

(ii) $\nph$ is a positive local Schur multiplier;

(iii) $f$ is equivalent to a positive definite function from $B(G)$.
\end{theorem}
\begin{proof}
(i)$\Rightarrow$(ii) is trivial.

(ii)$\Rightarrow$(iii)
By \cite[Corollary 4.5]{stt2011}, $\nph$ is equivalent to an
$\omega$-continuous function. By \cite[Proposition 7.3]{stt2011},
$f$ is equivalent to a continuous function $g : G\rightarrow
\bb{C}$. We may thus assume that $f$ is itself continuous.

Let $T(G) =
L^2(G)\hat{\otimes}L^2(G)$, where by $\hat{\otimes}$ we denote the
projective tensor product. The space $T(G)$ can be naturally
identified with the trace class on $L^2(G)$. We let $T(G)^+$ be the
cone in $T(G)$ corresponding to the positive trace class operators
under this identification; we have that $T(G)^+$ consists of all
elements of the form $\sum_{i=1}^{\infty} \xi_i\otimes \overline{\xi_i}$, with
$\sum_{i=1}^{\infty} \|\xi_i\|^2_2 < \infty$.
Let $P : T(G) \rightarrow A(G)$ be the contraction given by
$P(\xi\otimes\eta)(s) = (\lambda_s\xi,\overline{\eta})$, $\xi,\eta\in L^2(G)$.

Since $Nf$ is a positive local Schur multiplier, there exists an
increasing sequence $(X_n)_{n\in \bb{N}}$ of measurable subsets of
$G$ such that the set $G\setminus (\cup_{n\in\mathbb{N}}X_n)$ is null and $Nf|_{X_n\times X_n}$
is a positive Schur multiplier. Clearly, $\cup_{n\in \bb{N}}
L^2(X_n)$ is dense in $L^2(G)$. Since $G$ is amenable,
\cite[Lemma~7.2]{lau} shows that there exists a net
$(u_{\alpha})_{\alpha}$, with $u_{\alpha} = P(\xi_{\alpha}\otimes
\overline{\xi_{\alpha}})$, $\|\xi_{\alpha}\|\leq 1$,
which converges to the constant function
$1$ uniformly on compact subsets. Since $P$ is contractive and  the
uniform norm is dominated by the norm of $A(G)$, we can replace
$u_{\alpha}$ by a function of the form $v_{\alpha} =
P(\eta_{\alpha}\otimes\overline{\eta_{\alpha}})$, with
$\eta_{\alpha}$ having support in some $X_n$, $n\in \bb{N}$.

We have that $(Nf)(\eta_{\alpha}\otimes\overline{\eta_{\alpha}})\in T(G)^+$ for each $\alpha$. Applying the mapping $P$, we obtain
that $fv_{\alpha}\in A(G)^+$ for each $\alpha$.
Let $K = \{s_1,\dots,s_n\} \subseteq G$. We have that
$$(f(s_i s_j^{-1})v_{\alpha}(s_i s_j^{-1}))_{i,j}
\rightarrow_{\alpha} (f(s_i s_j^{-1}))_{i,j}.$$ Since the matrix
$(f(s_i s_j^{-1})v_{\alpha}(s_i s_j^{-1}))_{i,j}$ is positive for
each $\alpha$, it follows that $(f(s_i s_j^{-1}))_{i,j}$ is positive
as well. Thus, $f$ is a positive definite function. Since $f$ is continuous, we have that $f\in B(G)$ (see \cite{folland}).

(iii)$\Rightarrow$(i) Since $G$ is amenable, $B(G)$ coincides with
the algebra of all completely bounded multipliers of $A(G)$. The
fact that $Nf$ is a Schur multiplier follows from \cite{bf} (see
also \cite{spronk}). The proof of Theorem \ref{c_pd} now shows that
$Nf$ is a positive Schur multiplier.
\end{proof}

\begin{corollary}
(i) \ The space of all local Schur multipliers coincides with the
linear span of the cone of all local positive Schur multipliers.

(ii) The space of all local Schur multiplier of Toeplitz type is
strictly larger than the linear span of the cone of all local
positive Schur multipliers of Toeplitz type.

\end{corollary}
\begin{proof}
(i) follows from \cite[Theorem 3.6]{stt2011}, Theorem \ref{th_lps}
and a standard polarisation argument.

(ii) follows from Theorem \ref{th_toe} and the fact that the space
of local Schur multipliers of Toeplitz type is strictly larger than
that of Schur multipliers of Toeplitz type (see \cite[Remark
7.11]{stt2011}).
\end{proof}

It was shown in \cite{j1992} that if a continuous function $\nph$ of Toeplitz type defined on the direct product $G\times G$,
where $G$ is a locally compact group, is a Schur multiplier, then the functions $a,b : G\rightarrow \ell^2$ in the representation
$\nph(x,y) = (a(x),b(y))_{\ell^2}$, can be chosen to be continuous.
It is thus natural to ask the following questions:

\begin{question}
Let $X$ be a locally compact topological space equipped with a regular Borel measure.
Suppose that $\nph : X\times X\rightarrow \bb{C}$ is a continuous Schur multiplier.

(i) \ \ Do there exist continuous bounded functions $a,b : X\rightarrow \ell^2$ such that $\nph(x,y) = (a(x),b(y))_{\ell^2}$ for all $x,y\in X$?

(ii) \ If $\nph$ is moreover positive, can one choose a continuous bounded function
$a : X\rightarrow \ell^2$ such that $\nph(x,y) = (a(x),a(y))_{\ell^2}$ for all $x,y\in X$?

(iii) Assuming that $\nph$ is a local (resp. positive local) Schur
multiplier, can a similar choice be made with $a$ and $b$ (resp.
$a$) not necessarily bounded?
\end{question}

\section{Local operator multipliers}\label{s_lom}

In this section we introduce local operator multipliers, a
non-commutative version of local Schur multipliers, and characterise
them, generalising the characterisation of local Schur multipliers
given in \cite{stt2011}. The suitable setting for local operator
multipliers is that of von Neumann algebras, as opposed to the
setting of C*-algebras, which was used to define and study universal
multipliers in \cite{ks2006} and \cite{jtt2009}. We therefore start
with collecting some notions and results from \cite{jtt2009} in a
form convenient for our purposes.

Let $H$ and $K$ be Hilbert spaces and let $H^{\dd}$ be the dual
Banach space of $H$; note that $H^{\dd}$ is conjugate linear
isometric to $H$ via the map $\partial : H\rightarrow H^{\dd}$
sending $x\in H$ to the element $x^{\dd}\in H^{\dd}$ given by
$x^{\dd}(y) = (y,x)$, $y\in H$. If $T\in \cl B(H,K)$, we let
$T^{\dd} \in \cl B(K^{\dd},H^{\dd})$ be the dual operator of $T$. If
$\cl M\subseteq \cl B(H)$ is a von Neumann algebra, we denote by
$\cl M^o$ the opposite von Neumann algebra of $\cl M$; we have that
$\cl M^o\subseteq \cl B(H^{\dd})$ consists of the elements of the
form $a^{\dd}$, where $a\in \cl M$. In particular, $\cl B(H)^o = \cl
B(H^{\dd})$. By $H\otimes K$ we denote the Hilbert space tensor
product of $H$ and $K$. If $\cl M$ and $\cl N$ are von Neumann
algebras, we denote by $\cl M\bar{\otimes} \cl N$ the (spatial weak*
closed) von Neumann algebra tensor product. Thus, $\cl
B(H^{\dd}\otimes K) = \cl B(H)^o\bar{\otimes}\cl B(K)$.

We let $\theta : H^{\dd}\otimes K \rightarrow \cl C_2(H,K)$ be the
canonical isomorphism sending an elementary tensor $x^{\dd}\otimes
y$ to the rank one operator given by $\theta(x^{\dd}\otimes y)(z) =
(z,x)y$, $z\in H$. This allows us to equip $H^{\dd}\otimes K$ with
an \lq\lq operator'' norm:
$$\|\xi\|_{\op}\overset{def}=\|\theta(\xi)\|_{\op},\ \ \xi\in H^{\dd}\otimes K.$$

For $\nph \in \cl B(H^{\dd}\otimes K)$, we define $S_{\nph} : \cl
C_2(H,K)\rightarrow \cl C_2(H,K)$ to be the mapping given by
$S_{\nph}(\theta(\xi)) = \theta(\nph \xi)$, $\xi \in H^{\dd}\otimes
K$. We call $\nph$ an \emph{operator multiplier} if there exists $C
> 0$ such that $\|S_{\nph}(\theta(\xi))\|_{\op}\leq C
\|\theta(\xi)\|_{\op}$, for every $\xi \in H^{\dd}\otimes K$. If
$\nph$ is an operator multiplier, then the mapping $S_{\nph}$
extends by continuity to a mapping (denoted in the same way)
$S_{\nph} : \cl K(H,K)\rightarrow \cl K(H,K)$ and, by taking the
second dual, to a mapping $S_{\nph}^{**} : \cl B(H,K)\rightarrow \cl
B(H,K)$. An element $\nph \in \cl B(H^{\dd}\otimes K)$ will be
called a \emph{completely bounded operator multiplier}, or a
\emph{c.b. operator multiplier}, if $S_{\nph}$ is completely bounded
with respect to the operator space structure arising from the
inclusion $\cl C_2(H,K)\subseteq \cl K(H,K)$. If $\cl M\subseteq \cl
B(H)$ and $\cl N\subseteq \cl B(K)$ are von Neumann algebras, we
will denote by $\mcb(\cl M,\cl N)$ the collection of all c.b. operator multipliers in $\cl M^o\bar\otimes\cl N$ and call
its elements \emph{completely bounded $\cl M,\cl N$-multipliers}, or
\emph{c.b. $\cl M,\cl N$-multipliers}.
We note that $\mcb(\cl M,\cl N)$ is
a subalgebra of $\cl M^o\bar\otimes\cl N$.

We next recall \cite{bs1992} that the \emph{extended Haagerup tensor
product} $\cl B(K)\otimes_{\eh} \cl B(H)$ consists of the sums of
the form $\sum_{i=1}^{\infty}b_i\otimes a_i$, where $(b_i)_{i\in
\bb{N}}$ (resp. $(a_i)_{i\in \bb{N}}$) is a bounded row (resp.
column) operator. There exists a one-to-one correspondence between
the elements of $\cl B(K)\otimes_{\eh} \cl B(H)$ and the normal
completely bounded maps on $\cl B(H,K)$: to the element $u =
\sum_{i=1}^{\infty}b_i\otimes a_i \in \cl B(K)\otimes_{\eh} \cl
B(H)$, there corresponds the map $\Phi_u$ given by $\Phi_u(x) =
\sum_{i=1}^{\infty} b_i x a_i$, $x\in \cl B(H,K)$.

Let $\nph \in \mcb(\cl M,\cl N)$. The mapping $S_{\nph}^{**}$ is
normal and completely bounded; by the previous paragraph, there
exists a (unique) element $u_{\nph} \in \cl B(K)\otimes_{\eh}\cl
B(H)$, called the \emph{symbol} of $\nph$ \cite{jltt2009}, such that $S_{\nph}^{**}
= \Phi_{u_{\nph}}$. Moreover, \cite[Proposition 5.5]{jltt2009} shows
that $u_{\nph}\in \cl N\otimes_{\eh}\cl M$.
In particular, the map $S_{\nph}^{**}$ is $\cl N',\cl M'$-modular.

In the next proposition, we describe the elements $u\in \cl
N\otimes_{\eh}\cl M$ that are symbols of c.b. operator multipliers.

\begin{proposition}\label{c2pres}
The mapping $\Lambda : \nph\rightarrow \Phi_{u_{\nph}}$ is a
bijective homomorphism from $\mcb(\cl B(H),\cl B(K))$ onto the space
of all normal completely bounded maps on $\cl B(H,K))$ which leave
$\cl C_2(H,K)$ invariant.
\end{proposition}
\begin{proof}
Suppose that $\nph\in \mcb(\cl B(H),\cl B(K))$. The map
$\Phi_{u_{\nph}}$ is the unique normal extension of $S_{\nph} : \cl
C_2(H,K)\rightarrow \cl C_2(H,K)$ to $\cl B(H,K)$. It follows that
$\Phi_{u_{\nph}}$ preserves $\cl C_2(H,K)$.

Conversely, suppose that $\Phi$ is a normal completely bounded map
which leaves $\cl C_2(H,K)$ invariant. Let $\nph : H^{\dd}\otimes
K\rightarrow H^{\dd}\otimes K$ be the map given by $\nph\xi =
\theta^{-1}(\Phi(\theta(\xi)))$. Clearly, $\nph$ is a linear map. We
show that it has a closed graph. Suppose $\xi_k\rightarrow 0$ and
$\nph\xi_k\rightarrow \eta$ in the norm of $H^{\dd}\otimes K$. It follows
that $\|\theta(\xi_k)\|_{\op}\rightarrow 0$ and hence
$\|\Phi(\theta(\xi_k))\|_{\op}\rightarrow 0$. Thus,
$\|\theta(\nph\xi_k)\|_{\op}\rightarrow 0$ and hence $\eta = 0$.

It follows from the Closed Graph Theorem that $\nph\in \cl
B(H^{\dd}\otimes K)$. By its definition, $S_{\nph} = \Phi|_{\cl C_2(H,K)}$
and it follows that $\nph\in \mcb(\cl B(H),\cl B(K))$ and
$\Phi_{u_{\nph}} = \Phi$. The fact that $\Lambda$ is a homomorphism
is immediate from its definition.
\end{proof}

We recall that if $\cl A_1$ and $\cl A_2$ are C*-algebras, then the
Haagerup norm of an element $\omega$ of $\mathcal{A}_{1}\odot
\mathcal{A}_2$ is defined by
\begin{eqnarray*}
\|\omega\|_{\hh} = \inf\left\{\left\|\sum
a_ia_i^*\right\|^{\frac{1}{2}}\left\|\sum
b_i^*b_i\right\|^{\frac{1}{2}}:\omega=\sum a_i\otimes b_i\right\}.
\end{eqnarray*}
We also let \cite{ks2006}
\[\|\omega\|_{\ph}=\inf\left\{\left\|\sum a_ia_i^*\right\|^{\frac{1}{2}}\left\|\sum
b_ib_i^*\right\|^{\frac{1}{2}}:\omega=\sum a_i\otimes b_i\right\}.\]

Let $(\nph_{\nu})_{\nu}\subseteq \cl B(H^{\dd})\odot \cl B(K)$ be a
net and $\nph\in \cl B(H^{\dd}\otimes K)$. We write $\nph={\rm
m-}\lim_{\nu} \nph_{\nu}$ if the net $(\nph_{\nu})_{\nu}$ converges
semi-weakly to $\nph$ (that is, $\langle\nph_{\nu}(h_1\otimes
k_1),h_2\otimes k_2\rangle\rightarrow \langle\nph(h_1\otimes
k_1),h_2\otimes k_2\rangle$ for every $h_1,h_2\in H^{\dd}$ and $k_1,
k_2\in K$), and there exists $C
> 0$ such that $\|\nph_{\nu}\|_{\ph} \leq C$ for all $\nu$.

The following characterisation of c.b. $\cl M,\cl N$-multipliers
follows from \cite{jtt2009} and \cite{jltt2009}:

\begin{theorem}\label{th_prch}
An element $\nph\in \cl M^o\bar\otimes\cl N$ is a c.b. $\cl M,\cl
N$-multiplier if and only if there exists a net
$(\nph_{\nu})\subseteq \cl M^o\odot\cl N$ such that $\nph={\rm
m-}\lim_{\nu}\nph_{\nu}$.
\end{theorem}

We now introduce local operator multipliers as a non-commutative
version of local Schur multipliers. To motivate our definition,
recall that, in the commutative case, local Schur multipliers are
defined within the class of all measurable, in general unbounded,
functions on the direct product of two measure spaces. The natural
non-commutative analogue of this algebra is the set of all
operators, affiliated with the tensor product of two von Neumann
algebras. On the other hand, the non-commutative analogue of
measurable subsets are projections. We are thus naturally led to
define an \emph{$\cl M',\cl N'$-covering family} (where $\cl
M\subseteq \cl B(H)$ and $\cl N\subseteq \cl B(K)$ are von Neumann
algebras) as a family $\{p_n\otimes q_n\}_{n\in \bb{N}}$, where
$\{p_n\}_{n\in \bb{N}}\subseteq \cl M'$ and $\{q_n\}_{n\in
\bb{N}}\subseteq \cl N'$ are families of pairwise commuting
projections, such that $\vee_{n\in \bb{N}} p_n\otimes q_n = I$.

\begin{definition}
Given a von Neumann algebra $\mathcal{M}\subseteq\mathcal{B}(H)$ and
projections $\{p_i\}_{i\in I}\subseteq\mathcal{M}'$, we say that a
densely defined operator $A : H\rightarrow H$ is \emph{associated
with $\mathcal{M}$ with respect to $\{p_i\}$} if $p_iH\subseteq
\dom(A)$ and $p_iA, Ap_i\in \mathcal{M}p_i$ for all
$i\in\mathbb{N}$.
\end{definition}

The set of all such operators will be denoted by $\Assoc\mathcal{M}_{\{p_i\}}$.

\begin{definition}\label{d_lop}
Let $\cl M\subseteq \cl B(H)$ and $\cl N\subseteq \cl B(K)$ be von
Neumann algebras. An element $\nph \in \Aff(\cl M^o\bar\otimes\cl
N)$ will be called a \emph{local $\cl M, \cl N$-multiplier} if there
exists an $\cl M', \cl N'$-covering family $\{p_n\otimes q_n\}_{n\in
\bb{N}}$ such that $\nph(p_n^{\dd}\otimes q_n) \in \mcb(\cl Mp_n,
\cl Nq_n)$ for every $n\in \bb{N}$.
\end{definition}

\begin{lemma}\label{l_lm}
Let $\cl M\subseteq \cl B(H)$ and $\cl N\subseteq \cl B(K)$ be von
Neumann algebras and $\nph \in \cl M^o\bar\otimes\cl N$.

(i) If $p_1,p_2\in \cl M'$ and $q_1,q_2\in \cl N'$ are projections,
$p_1\leq p_2$, $q_1\leq q_2$, and $\nph(p_2^{\dd}\otimes q_2) \in
\mcb(\cl Mp_2,\cl Nq_2)$ then $\nph(p_1^{\dd}\otimes q_1) \in
\mcb(\cl Mp_1,\cl Nq_1)$.

(ii) If $(e_i)_{i=1}^n \subseteq \cl M'$, $(f_j)_{i=1}^m \subseteq
\cl N'$ are sequences of pairwise orthogonal projections such that
$\vee_{i=1}^n e_i = I$, $\vee_{j=1}^m f_j = I$ and
$\nph(e_i^{\dd}\otimes f_j)\in \mcb(\cl Me_i,\cl Nf_j)$ for each
$i,j$, then $\nph\in \mcb(\cl M,\cl N)$.
\end{lemma}
\begin{proof}
(i) follows from the fact that $S_{\nph(p_1^{\dd}\otimes q_1)} = S_{\nph(p_2^{\dd}\otimes q_2)}|_{\cl C_2(p_1 H, q_1 K)}$.

(ii) Since $S_{\nph} = \sum_{i,j} S_{\nph(p_i^{\dd}\otimes q_j)}$, we have that
$\|S_{\nph}\|_{\cb} \leq \sum_{i,j}\|S_{\nph(p_i^{\dd}\otimes q_j)}\|_{\cb}$.
\end{proof}

\begin{proposition}\label{l_orthog}
Let $\cl M\subseteq \cl B(H)$ and $\cl N\subseteq \cl B(K)$ be von
Neumann algebras and suppose that $\{e_i\}\subseteq \cl M'$ and
$\{f_j\}\subseteq \cl N'$ are at most countable families of pairwise
orthogonal projections. Let $\cl E$ be the linear span of
$\cup_{i,j} f_j\cl B(H,K)e_i$, and $\Phi : \cl E\rightarrow \cl
B(H,K)$ be a linear map. The following are equivalent:

(i) $\Phi$ leaves $f_j\cl B(H,K)e_i$ invariant, and the restriction
$\Phi_{i,j} \stackrel{def}{=} \Phi|_{f_j\cl B(H,K)e_i} : f_j\cl
B(H,K)e_i\rightarrow f_j\cl B(H,K)e_i$ is $f_j\cl N'f_j, e_i\cl M'
e_i$-modular, normal and completely bounded;

(ii) there exist families $\{a_k\}_{k\in
\bb{N}}\subseteq\Assoc\mathcal {M}_{\{e_i\}}$ and
$\{b_k\}_{k\in\bb{N}}\subseteq\Assoc\mathcal {N}_{\{f_j\}}$ such
that $(e_ia_k)_{k\in \bb{N}}$ defines a bounded column operator for
each $i$, $(b_kf_j)_{k\in \bb{N}}$ defines a bounded row operator
for each $j$, and $\Phi(x) = \sum_{k=1}^{\infty} b_k x a_k$, for all
$x\in \cl E$.
\end{proposition}
\begin{proof}
(ii)$\Rightarrow$(i) Suppose that $x = f_j x e_i$ for some $i,j\in
\bb{N}$. We have that $a_k = \sum_{i=1}^{\infty} a_k e_i$, where the
sum converges pointwise on $\cup_{i=1}^{\infty} e_i H$. A similar
formula holds for $b_k$. By assumption, for every $i$ (resp. $j$)
and and for every $k$, there exists $\tilde{a}_{k,i}\in \cl M$
(resp. $\tilde{b}_{k,j}\in \cl N$) such that $e_ia_k =
e_i\tilde{a}_{k,i} $ (resp. $b_kf_j = \tilde{b}_{k,j}f_j $). Let
$c\in e_i\cl M'e_i$ and $d\in f_j\cl N'f_j$. We have that
\begin{eqnarray*}
\Phi(dxc) & = & \sum_{k=1}^{\infty} b_k (f_jdf_jxe_ice_i)a_k =
\sum_{k=1}^{\infty}\tilde{b}_{k,j} (f_jdf_j x e_ice_i) \tilde{a}_{k,i}\\
& = & \sum_{k=1}^{\infty}f_jdf_j \tilde{b}_{k,j} f_jxe_i
\tilde{a}_{k,i}(e_ice_i) = d\left(\sum_{k=1}^{\infty}
b_{k}xa_k\right)c.
\end{eqnarray*}

These identities show that $\Phi$ leaves $f_j \cl B(H,K)e_i$
invariant, and that its restriction to $f_j \cl B(H,K)e_i$ is a
normal completely bounded $f_j\cl N'f_j,e_i\cl M'e_i$-modular map.

(i)$\Rightarrow$(ii) For each $i$ and $j$, let $A_{i,j} =
(a_{i,j}^k)_{k\in \bb{N}}\in M_{\infty,1}(\cl Me_i)$ be a bounded
column operator and $B_{i,j} = (b_{i,j}^k)_{k\in \mathbb{N}}\in
M_{1,\infty}(\cl Nf_j)$ be a bounded row operator such that
$$\Phi(x) = \sum_{k=1}^{\infty} b_{i,j}^k x a_{i,j}^k = B_{i,j} (x\otimes 1) A_{i,j}, \ \ x\in f_j \cl B(H,K)e_i.$$
Assume, without loss of generality, that $\|A_{i,j}\| =
\|B_{i,j}\|$, and let $\alpha_{i,j} = \|A_{i,j}\|^2$. By \cite[Lemma
3.5]{stt2011}, there exist vectors $r_i = (r_i(l))_{l\in \bb{N}},s_j
= (s_j(l))_{l\in \bb{N}}\in \ell^2$ such that
$(r_i,\overline{s_j})_{\ell^2} = \alpha_{i,j}$.

Let $\bb{N}_1$, $\bb{N}_2$, $\bb{N}_3$ and $\bb{N}_4$ be copies of
$\bb{N}$, set $\Lambda = \bb{N}_1\times \bb{N}_2\times
\bb{N}_3\times \bb{N}_4$ and equip $\Lambda$ with the lexicographic
order, where each $\bb{N}_s$, $s = 1,2,3,4$, is given its natural
order. Let $A$ (resp. $B$) be the column (resp. row) operator given
by
$$A = \left(\frac{i r_i(l)}{j\sqrt{\alpha_{i,j}}} a_{i,j}^k\right)_{(i,j,k,l)\in \Lambda} \ \ \
(\mbox{resp. } B = \left(\frac{j s_j(l)}{i\sqrt{\alpha_{i,j}}}
b_{i,j}^k\right)_{(i,j,k,l)\in \Lambda}).$$ We note that $A$ (resp.
$B$) does not necessarily define a bounded row (resp. column)
operator, but it can be regarded as a linear operator densely
defined on $[\cup_{i\in\bb{N}}e_i H]$ (resp. $[\cup_{j\in \bb{N}} f_j K]$).

We have that the non-zero entries of $e_nA$ are the elements of the
family $\left(\frac{n r_n(l)}{j\sqrt{\alpha_{n,j}}}
a_{n,j}^k\right)_{j,k,l\in \bb{N}}$, and hence
$$\|e_nA\|\leq n\|r_n\|_2 \left(\sum_{j=1}^{\infty} \frac{1}{j^2}\right)^{1/2}.$$
Similarly,
$$\|Bf_n\|\leq n\|s_n\|_2 \left(\sum_{j=1}^{\infty} \frac{1}{j^2}\right)^{1/2}.$$

Suppose that $x = f_{j_0}x e_{i_0}$. Then
\begin{eqnarray*}
B(x\otimes 1)A & = & \sum_{l,k,i,j}\frac{j
s_j(l)}{i\sqrt{\alpha_{i,j}}} b_{i,j}^k x \frac{i
r_i(l)}{j\sqrt{\alpha_{i,j}}} a_{i,j}^k
=  \sum_{i,j,k} ij\alpha_{i,j} \frac{1}{i\sqrt{\alpha_{i,j}}} b_{i,j}^k  x \frac{1}{j\sqrt{\alpha_{i,j}}} a_{i,j}^k\\
& = & \sum_{k=1}^{\infty}
b_{i_0,j_0}^k x a_{i_0,j_0}^k =
B_{i_0,j_0}(x\otimes 1)A_{i_0,j_0} = \Phi(x).
\end{eqnarray*}
The claim now follows by choosing any enumeration of $\Lambda$.
\end{proof}

\begin{lemma}\label{ip}
Let $H$ and $K$ be Hilbert spaces and $\varphi\in\mathcal{B}(H^{\dd}\otimes K)$.
Then
$$(S_{\varphi}(\theta(\xi))h_1,h_2)_K = (\varphi\xi,h_1^{\dd}\otimes
h_2)_{H^{\dd}\otimes K}, \ \ \xi\in H^{\dd}\otimes K, h_1^{\dd}\in
H^{\dd}, h_2\in K.$$
\end{lemma}
\begin{proof}
Using the identity
$$\tr\left(T\theta\left(h_1^{\dd}\otimes h_2\right)\right) = (Th_1,h_2)_K,\  T\in\mathcal{C}_2(H,K),$$
we see that
\begin{eqnarray*}
\left(\varphi\xi,h_1^{\dd}\otimes h_2\right)_{H^{\dd}\otimes K} & = &
\left(\theta\left(\varphi\xi\right),\theta\left(h_1^{\dd}\otimes
h_2\right)\right)_{\cl C_2} = \tr\left(\theta\left(\varphi\xi\right)\theta\left(h_1^{\dd}\otimes h_2\right)\right)\\
& = & \left(\theta\left(\varphi\xi\right)h_1,h_2\right)_K = \left(
S_{\varphi}\left(\theta(\xi)\right)h_1,h_2\right)_K.
\end{eqnarray*}
\end{proof}

\begin{theorem}\label{th_chavnmu}
Let $\varphi\in\Aff (\mathcal{M}^o\bar{\otimes}\mathcal{N})$.  Then
the following are equivalent:

(i) $\varphi$ is a local $\cl M,\cl N$-multiplier;

(ii) there exist increasing sequences $(p_n)_{n\in
\bb{N}}\subseteq\cl M'$ and $(q_n)_{n\in \bb{N}}\subseteq\cl N'$ of projections
such that $\vee p_n\otimes q_n=I$ and $\nph(p_n^{\dd}\otimes q_n)\in
\mcb(\cl M p_n, \cl N q_n)$ for every $n\in \bb{N}$;

(iii) there exist families $(e_i)_{i\in \bb{N}}\subseteq \cl M'$ and
$(f_j)_{j\in \bb{N}}\subseteq\mathcal{N}'$ of mutually orthogonal
projections such that $\vee_{i\in \bb{N}} e_i = I$ and $\vee_{j\in
\bb{N}} f_j = I$, and a net
$(\nph_{\nu})_{\nu}\subseteq\Assoc\mathcal{M}^o\bar{\otimes}\mathcal{N}_{\{e_i^{\dd}\otimes
f_j\}}$ such that $\nph_{\nu}(e_i^{\dd}\otimes f_j)\in \cl M^o
e_i^{\dd}\odot \cl N f_j$ and $\nph(e_i^{\dd}\otimes f_j) = {\rm
m-}\lim_{\nu} \nph_{\nu}(e_i^{\dd}\otimes f_j)$, for all $i,j$.
\end{theorem}
\begin{proof}
(i)$\Rightarrow$(iii) Let $\{p_n\otimes q_n\}_{n\in \bb{N}}\subseteq
\cl M'\otimes \cl N'$ be a covering family of projections such that
$\nph(p_n^{\dd}\otimes q_n)\in \mcb(\cl M p_n, \cl N q_n)$. Let $\cl
C\subseteq \cl M'$ (resp. $\cl D\subseteq \cl N'$) be a masa in $\cl
M'$ (resp. $\cl N'$) containing the family $\{p_n\}_{n\in\bb{N}}$
(resp. $\{q_n\}_{n\in \bb{N}}$). Since the Hilbert spaces $H$ and
$K$ are assumed to be separable, $\cl C$ (resp. $\cl D$) is
isomorphic to the multiplication masa of $L^{\infty}(X,\mu)$ (resp.
$L^{\infty}(Y,\sigma)$), for some standard measure space $(X,\mu)$
(resp. $(Y,\sigma)$). Under this isomorphism, $p_n$ and $q_n$
correspond to some measurable sets $\alpha_n\subseteq X$ and
$\beta_n\subseteq Y$, respectively. Lemma 3.4 from \cite{stt2011}
implies that there exist families $\{X_i\}_{i\in \bb{N}}$ and
$\{Y_j\}_{j\in \bb{N}}$ of pairwise disjoint measurable subsets  of
$X$ and $Y$, respectively, such that $\cup_{i\in \bb{N}} X_i = X$,
$\cup_{j\in \bb{N}} Y_j = Y$, and each $X_i\times Y_j$ is contained
in the union of finitely many of the sets $\alpha_n\times \beta_n$.
Let $e_i$ (resp. $f_j$) be the projection in $\cl M'$ (resp. $\cl
N'$) corresponding to the set $X_i$ (resp. $Y_j$). Then $\vee_{i\in
\bb{N}} e_i = I$ and $\vee_{j\in \bb{N}} f_j = I$.

Fix $i$ and $j$; then $e_i\otimes f_j = \vee_{n\in \bb{N}} p_n e_i
\otimes q_n f_j$ where, in view of the properties of the families
$\{X_i\}$ and $\{Y_j\}$, the span is over a finite set of indices.
By Lemma \ref{l_lm} (i), $\nph(e_i^{\dd}p_n^{\dd} \otimes f_jq_n
)\in \mcb(\cl M e_ip_n, \cl N f_jq_n)$ for all $i,j,n$, and by Lemma
\ref{l_lm} (ii), $\nph(e_i^{\dd} \otimes f_j)\in \mcb(\cl M e_i, \cl
N f_j)$ for all $i,j$.

Let $\cl E = [\cup_{i,j} f_j \cl B(H,K) e_i]$ and let $\Phi : \cl
E\rightarrow \cl E$ be the map whose restriction to $f_j \cl B(H,K)
e_i$ coincides with $S^{**}_{\nph(e_i^{\dd}\otimes f_j)}$. By Proposition
\ref{l_orthog}, there exist operators $(a_k)_{k\in
\bb{N}}\subseteq\Assoc\mathcal {M}_{\{e_i\}_{i\in\mathbb{N}}}$ and
$(b_k)_{k\in\bb{N}}\subseteq\Assoc\mathcal
{N}_{\{f_j\}_{j\in\mathbb{N}}}$ such that $(e_ia_k)_{k\in \bb{N}}$
defines a bounded column operator for each $i$, $(b_kf_j)_{k\in
\bb{N}}$ defines a bounded row operator for each $j$, and $\Phi(x) =
\sum_{k=1}^{\infty} b_k x a_k$, for all $x\in \cl E$.

Let $\nph_N = \sum_{k=1}^N a_k^{\dd}\otimes b_k$. Then
$\nph_N(e_i^{\dd}\otimes f_j)$ belongs to $(\cl M^o e_i^{\dd})\odot
(\cl N f_j)$ for all $i,j$, and
$$\sup_{N\in \bb{N}}\|\nph_N(e_i^{\dd}\otimes f_j)\|_{\ph}
\leq \sup_{N\in \bb{N}}\left\|\sum_{k=1}^N b_k f_j\otimes e_i a_k
\right\|_{\hh} \leq \|(e_i a_k)_{k\in \bb{N}}\|\|(b_k f_j)_{k\in
\bb{N}}\|.$$ For all $\xi \in (e_i^{\dd}H^{\dd}) \otimes (f_j K)$
and all $h^{\dd}\in H^{\dd}$, $k\in K$, we have, by Lemma \ref{ip},
that
\begin{equation}\label{eq_idip}
(\Phi(\theta(\xi))h,k)_K = (\nph(e_i^{\dd}\otimes f_j)\xi,
h^{\dd}\otimes k)_{H^{\dd}\otimes K}.
\end{equation}
On the other hand, if $\Phi_N = S^{**}_{\nph_N}$, then we have that
\begin{equation}\label{eq_phin}
\Phi_N(\theta(\xi))\rightarrow_{N\rightarrow \infty} \Phi(\theta(\xi)) \mbox{ weakly, for all }
\xi \in (e_i^{\dd}H^{\dd}) \otimes (f_j K).
\end{equation}
It follows from (\ref{eq_idip}) and (\ref{eq_phin}) that
$(\nph_N(e_i^{\dd}\otimes f_j))_{N=1}^{\infty}$ converges
semi-weakly to $\nph(e_i^{\dd}\otimes f_j)$, for all $i,j$. Thus,
$\nph(e_i^{\dd}\otimes f_j) ={\rm m-}\lim_N \nph_N(e_i^{\dd}\otimes
f_j)$ and (iii) is established.

(iii)$\Rightarrow$(ii)
Fix $i$ and $j$ and let $\psi = \nph(e_i^{\dd}\otimes f_j)$ and $\psi_{\nu} = \nph_{\nu}(e_i^{\dd}\otimes f_j)$.
Then $(\psi_{\nu}) $ converges semi-weakly to
$\psi$ and $D = \sup_{\nu}\|\psi_{\nu}\|_{\ph} < +\infty$.
A standard argument shows that
$\|S_{\psi_{\nu}}(\theta(\xi))\| \leq D \|\theta(\xi)\|$ for all $\nu$ and all $\xi\in H^{\dd}\otimes K$.
Lemma \ref{ip} and the lower semi-continuity of the norm in the weak operator topology
imply that
$\|S_{\psi}(\theta(\xi))\| \leq D \|\theta(\xi)\|$ for all $\xi\in H^{\dd}\otimes K$; thus,
$\|S_{\psi}\|\leq D$. A similar argument shows that we have the stronger inequality
$\|S_{\psi}\|_{\cb}\leq D$.
Let $p_n = \vee_{i=1}^n e_i$ and
$q_n = \vee_{j=1}^n f_j$. The claim now follows from Lemma
\ref{l_lm} (ii).

(ii)$\Rightarrow$(i) follows from the fact that $(p_n\otimes
q_n)_{n\in \bb{N}}$ is a covering family.
\end{proof}

We next include analogous versions of some of the previous results for the case where
the respective projections are central. The first one is a \lq\lq local'' version of the
well-known representation representation theorem for completely bounded bimodular maps (see \cite{haagerup,smith1991}).

\begin{proposition}\label{p_clm}
Let $\cl M\subseteq \cl B(H)$ and $\cl N\subseteq \cl B(K)$ be von
Neumann algebras, $(P_n)_{n\in \bb{N}}\subseteq \cl M\cap \cl M'$
and $(Q_n)_{n\in \bb{N}}\subseteq \cl N\cap \cl N'$ be increasing
sequences of projections, $\cl E = [\cup_{n\in \bb{N}}Q_n\cl
B(H,K)P_n]$ and $\Phi : \cl E\rightarrow \cl B(H,K)$ be a linear
map. The following are equivalent:

(i) \ the restriction of $\Phi$ to $ Q_n\cl B(H,K)P_n$ is completely bounded, normal and
$ (\cl NQ_n)', (\cl MP_n)'$-modular;

(ii) there exist families $(a_k)_{k\in \bb{N}}\subseteq\mathcal{M}$
and $(b_k)_{k\in \bb{N}}\subseteq\mathcal{N}$ such that\linebreak
$(P_na_k)_{k\in \bb{N}}$ (resp. $(b_kQ_n)_{k\in \bb{N}}$) is a
bounded column (resp. row) operator for every $n$ and $\Phi(x) =
\sum_{k=1}^{\infty} b_k x a_k$, for every $x \in \cl E$.
\end{proposition}
\begin{proof}
(i)$\Rightarrow$(ii) Let $e_1 = P_1$ (resp. $f_1 = Q_1$) and $e_i =
P_{i+1} - P_i$ (resp. $f_j = Q_{j+1} - Q_j$), $i\geq 2$ (resp.
$j\geq 2$). It is clear that $\Phi|_{f_j\cl B(H,K)e_i} : f_j\cl
B(H,K)e_i\rightarrow \cl B(H,K)$ is completely bounded and $\cl
N'f_j, \cl M'e_i$-modular. Let
$$A = \left(\frac{i r_i(l)}{j\sqrt{\alpha_{i,j}}} a_{i,j}^k\right)_{(i,j,k,l)\in \Lambda} \mbox{and } \
B = \left(\frac{j s_j(l)}{i\sqrt{\alpha_{i,j}}}
b_{i,j}^k\right)_{(i,j,k,l)\in \Lambda}$$ be the operators from the
proof of Proposition \ref{l_orthog}. Since the projections $P_n$ and
$Q_n$, $n\in \bb{N}$, are central, we have that the entries of $A$
(resp. $B$) belong to $\cl M$ (resp. $\cl N$). The estimates from
the proof of Proposition \ref{l_orthog} show that
$$\|P_nA\|^2\leq \sum_{d=1}^n d^2\|r_d\|_2^2 \sum_{j=1}^{\infty} \frac{1}{j^2} \mbox{ and }
\|BQ_n\|^2\leq \sum_{d=1}^n d^2\|s_d\|_2^2 \sum_{j=1}^{\infty}
\frac{1}{j^2}.$$ The conclusion follows.

(ii)$\Rightarrow$(i) follows from the fact that since the
projections are central, the conditions in (ii) hold for the
families $\{P_i-P_{i-1}\}_{i\in\mathbb{N}},
\{Q_j-Q_{j-1}\}_{j\in\mathbb{N}}$ of mutually orthogonal projections
and so the result follows as in the proof of the implication
(ii)$\Rightarrow$(i) of Proposition \ref{l_orthog}.
\end{proof}

Call a local $\cl M,\cl N$-multiplier \emph{central} if the covering
family $\{p_n\otimes q_n\}_{n\in \bb{N}}$ associated with $\nph$ as
in Definition \ref{d_lop} can be chosen from the centre of $\cl
M'\bar{\otimes}\cl N'$.

\begin{corollary}\label{c_ccbm}
Let $\varphi\in\Aff\mathcal{M}^o\overline{\otimes}\mathcal{N}$.
Then the following are equivalent:

(i) $\varphi$ is a central local $\cl M,\cl N$-multiplier;

(ii) there exist a net $(\nph_{\nu})\subseteq  \Aff \cl
M^o\bar\otimes\cl N$ and increasing sequences of projections
$(P_n)_{n\in \bb{N}}\subseteq\cl M\cap \cl M'$ and $(Q_n)_{n\in
\bb{N}}\subseteq\cl N\cap \cl N'$ such that
$\vee_{n\in\mathbb{M}}P_n\otimes Q_n=I$,
$\nph_{\nu}(P_n^{\dd}\otimes Q_n)\in \cl M^o P_n^{\dd}\odot \cl N
Q_n$ and $\varphi(P_n^{\dd}\otimes Q_n) = {\rm m-}\lim_{\nu}
\nph_{\nu}(P_n^{\dd}\otimes Q_n)$ for every $n\in \bb{N}$.

\end{corollary}
\begin{proof}
(ii)$\Rightarrow$(i) follows from Theorem \ref{th_chavnmu}.

(i)$\Rightarrow$(ii) It is easy to see that the
operators $\varphi_N$ from the proof of Theorem \ref{th_chavnmu} can,
under the assumption of the corollary,
be chosen from $\cl M^o\odot\cl N$. The conclusion follows by
letting $P_n = \vee_{i=1}^n e_i$ and $Q_n = \vee_{j=1}^n f_j$, where
$(e_i)$ and $(f_j)$ are the sequences of projections from Theorem
\ref{th_chavnmu}.
\end{proof}

\section{Positive local operator multipliers}\label{s_pos}

In this section, we study completely positive
local operator multipliers. The main result is the characterisation Theorem \ref{th_cplm}. Throughout
this section, we fix a von Neumann algebra $\cl M$.

\begin{definition}
Let $\mathcal{M}\subseteq\mathcal{B}(H)$ be a von Neumann algebra
and $\varphi\in\mathcal{M}^o\overline{\otimes}\mathcal{M}$. We say
that $\varphi$ is a \emph{completely positive
$\mathbf{\mathcal{M}}$-multiplier} if the map
$S_{\varphi}:\mathcal{C}_2(H)\rightarrow\mathcal{C}_2(H)$, given by
\[S_{\varphi}\left(\theta\left(\xi\right)\right)=\theta(\varphi\xi),\
\xi\in H^{\dd}\otimes H,\] is completely positive and bounded in
$\|\cdot\|_{\op}$.
\end{definition}

Let
\[\mathcal{P}(\mathcal{M})=\left\{\sum_{k=1}^N b_k^{\dd}\otimes
b_k^* : b_k\in\mathcal{M}, N\in\mathbb{N}.\right\}\subseteq \cl M^o
\odot \cl M.\] It is clear that $\mathcal{P}(\mathcal{M})$ is a
cone, and it is easy to verify that if
$\psi\in\mathcal{P}(\mathcal{M})$ then the map $S_{\psi}$ is
completely positive; thus, every element of $\cl P(\cl M)$ is a
completely positive $\cl M$-multiplier. In the next theorem, we show
that completely positive $\cl M$-multipliers can be approximated by
elements of $\cl P(\cl M)$.

\begin{theorem}\label{th_monoc}
Let $\varphi\in\mathcal{M}^o\overline{\otimes}\mathcal{M}$.  Then
$\varphi$ is a completely positive $\mathcal{M}$-multiplier if and
only if there exists a net $(\varphi_{\nu})_{\nu\in J}\subseteq \cl
P(\cl M)$ such that \mbox{$\varphi= {\rm m-}\lim_{\nu} \nph_{\nu}$.}
\end{theorem}
\begin{proof}
Suppose that $\varphi$ is a completely positive
$\mathcal{M}$-multiplier. By definition, $S_{\varphi}$ is bounded
and completely positive and thus, by the remarks before Proposition
\ref{c2pres}, it is $\mathcal{M}'$-bimodular. There exists a family
$(a_i)_{i=1}^{\infty}\subseteq\mathcal{M}$ of operators that defines
a bounded column operator $V$ such that
$$S_{\varphi}(x)=\sum_{i=1}^{\infty}a_i^*xa_i, \ \ \ \
x\in\mathcal{K}(H).$$ Let $\varphi_N = \sum_{i=1}^Na_i^{\dd}\otimes
a_i^*$, and note that $\varphi_N\in\mathcal{P}(\mathcal{M})$,
$N\in\mathbb{N}$. Now
\begin{equation}\label{eq_wc}S_{\varphi_N}(\theta(\xi))\rightarrow_{N\rightarrow\infty}
S_{\varphi}(\theta(\xi))\end{equation} weakly for all $\xi\in
H^{\dd}\otimes H$. It follows from Lemma \ref{ip} that
$(\varphi_N)_{N\in\mathbb{N}}$ converges semi-weakly to $\varphi$.
A standard estimate shows that
$$\sup_{N\in\mathbb{N}}\left\|\varphi_{N}\right\|_{\ph}  \leq\|V\|^2.$$
Thus, $\nph = {\rm m-}\lim_{\nu} \nph_{\nu}$.

Conversely, suppose that there exists a net
$(\varphi_{\nu})_{\nu\in J}\subseteq\mathcal{P}(\mathcal{M})$ such
that $(\varphi_{\nu})_{\nu\in J} $ converges semi-weakly to
$\varphi$ and $D = \sup_{\nu}\|\varphi_{\nu}\|_{\ph} < +\infty$.
As in the proof of the implication (iii)$\Rightarrow$(ii) of Theorem \ref{th_chavnmu}, we can see that
$\|S_{\varphi}\|\leq D$.

To obtain the complete positivity, suppose that
$(\theta(\xi_{ij}))_{i,j=1}^l$ is a positive element of
$\mathcal{C}_2(H^l)$, where $\xi_{ij}\in H^{\dd}\odot H$, $i,j =
1,\dots, l$. If $h = (h_1,\dots,h_l)\in H^l$ then

\begin{eqnarray*}
0&\leq&
\left(S_{\varphi_{\nu}}^{(l)}\left(\left(\theta\left(\xi_{ij}\right)\right)_{i,j=1}^l\right)h,h\right)_{H^n}\\
& = &
\sum_{i=1}^l\left(\sum_{j=1}^l\sum_{k=1}^{N_{\nu}}(a_k^{\nu})^*\theta(\xi_{ij})a_k^{\nu}h_j,h_i\right)_H
\\&=&\sum_{i,j=1}^l\left(\theta\left(\sum_{k=1}^{N_{\nu}}\left(a_k^{\nu}\right)^{\dd}
\otimes (a_k^{\nu})^*\xi_{ij}\right)h_j,h_i\right)_H
\\&=&\sum_{i,j=1}^l\left(\theta\left(\sum_{k=1}^{N_{\nu}}\left(a_k^{\nu}\right)^{\dd}
\otimes\left(a_k^{\nu}\right)^*\xi_{ij}\right),\theta\left(h_j^{\dd}\otimes
h_i\right)\right)_{\cl C_2(H)}
\\&=&\sum_{i,j=1}^l\left(\varphi_{\nu}\left( \xi_{ij}\right),h_j^{\dd}\otimes h_i\right)_{H^{\dd}\otimes
H}
\\&\rightarrow& \sum_{i,j=1}^l\left(\varphi\left( \xi_{ij}\right),h_j^{\dd}\otimes
h_i\right)_{H^{\dd}\otimes H} = \left(
S_{\varphi}^{(l)}((\theta(\xi_{ij}))_{i,j=1}^{l})h,h\right)_H
\end{eqnarray*}
and hence the map $S_{\varphi}$ is completely positive on the
*-algebra $\mathcal{F}(H)$ of all finite rank operators.

If $(\theta(\xi_{ij}))_{i,j=1}^l\in M_l(\mathcal{C}_2(H))^+$
then $(\theta(\xi_{ij}))$ can be approximated by positive matrices of finite rank operators in
the Hilbert-Schmidt norm, and hence in the operator norm. It follows from the
previous arguments that $S_{\varphi}$ is completely positive on
$\mathcal{C}_2(H)$. Thus $\varphi$ is a completely positive $\mathcal{M}$-multiplier.
\end{proof}

We next introduce the non-commutative version of positive local
Schur multipliers.

\begin{definition}\label{d_lcps}
Let $\mathcal{M}\subseteq\mathcal{B}(H)$ be a von Neumann algebra
and $\varphi\in\Aff(\mathcal{M}^o\overline{\otimes}\mathcal{M})$. We
say that $\varphi$ is a \emph{completely positive local
$\mathbf{\mathcal{M}}$-multiplier} if

(i) \ there exists an increasing sequence $(p_n)_{n=1}^{\infty}
\subseteq\mathcal{M}'$ of projections such that
$\vee_{n\in\mathbb{N}}p_n=I$ and
$\varphi\in\Assoc\mathcal{M}^o\overline{\otimes}\mathcal{M}_{\{p_n^{\dd}\otimes
p_n\}_{n\in\mathbb{N}}}$;

(ii) $\varphi(p_n^{\dd}\otimes p_n)$ is a completely positive
$\mathcal{M}p_n$-multiplier.
\end{definition}

We note that every completely positive local
$\mathcal{M}$-multiplier is a local $\mathcal{M}$-multiplier.
We will call the sequence $(p_n)_{n\in \bb{N}}$ of projections from Definition \ref{d_lcps}
an \emph{implementing sequence} for $\nph$.

\begin{theorem}\label{th_cplm}
An operator
$\varphi\in\Aff(\mathcal{M}^o\overline{\otimes}\mathcal{M})$ is a
completely positive local $\mathcal{M}$-multiplier with an implementing sequence $(p_n)_{n\in \bb{N}}$
if and only if
there exists a net $(\varphi_{\nu})_{\nu\in
J}\subseteq\Aff\mathcal{M}^o\odot\Aff\mathcal{M}$ such that
$\varphi_{\nu}(p_n^{\dd}\otimes p_n)\in
\mathcal{P}(\mathcal{M}p_n)$, and
$\varphi(p_n^{\dd}\otimes p_n)={\rm m-}\lim_{\nu}
\varphi_{\nu}(p_n^{\dd}\otimes p_n)$, for each $n\in \bb{N}$.
\end{theorem}
\begin{proof}
Suppose that $\varphi$ is a completely positive local
$\mathcal{M}$-multiplier with an implementing sequence $(p_n)_{n\in \bb{N}}$.
By definition, $\vee_{n\in\mathbb{N}}p_n=I$ and, if
$H_n=p_nH$, the map $S_{\varphi}|_{\mathcal{C}_2(H_n)}$ is bounded,
completely positive and $p_n\mathcal{M}'p_n$-bimodular.

Since the map is bounded on $\mathcal{C}_2(H_n)$, it can be extended
to $\mathcal{K}(H_n)$, and this extension preserves the bimodularity
and complete positivity. Thus, by Theorem \ref{repv}, there exists a
family $\{a_k\}_{k=1}^{\infty}$ of closable operators affiliated
with $\mathcal{M}$ that, as noted in Remark \ref{r_rc}, are such that $a_k
p_n\in\mathcal{M}p_n$ and \mbox{$S_{\varphi}(x)=\sum_{k=1}^{\infty}
a_k^*xa_k,$} for $x\in\cup_{n=1}^{\infty}\mathcal{K}(H_n)$. Recall
that $(a_1p_n,a_2p_n,\dots)^t$ is a bounded column operator, say
$V_n$. We define $\varphi_N=\sum_{k=1}^Na_k^{\dd}\otimes a_k^*$,
$n\in \bb{N}$. Clearly,
$\varphi_N\in\Aff\mathcal{M}^o\odot\Aff\mathcal{M}$ for each
$N\in\mathbb{N}$ and since $a_kp_n\in \mathcal{M}p_n$, it follows
that $\varphi_N(p_n^{\dd}\otimes
p_n)\in\mathcal{P}(\mathcal{M}p_n)$.

Analogously to (\ref{eq_wc}), we see that the sequence
$\left(S_{\varphi_N}|_{\mathcal{C}_2(H_n)}\left(\theta\left(\xi\right)\right)\right)_{N=1}^{\infty}$
converges weakly to
$S_{\varphi}|_{\mathcal{C}_2(H_n)}\left(\theta\left(\xi\right)\right)$.
By Lemma \ref{ip}, $\left(\varphi_N\left(p_n^{\dd}\otimes
p_n\right)\right)_{N\in\mathbb{N}}$ converges semi-weakly to
$\varphi\left(p_n^{\dd}\otimes p_n\right)$. As before, one can easily see that
$$\sup_{N\in\mathbb{N}}\left\|\varphi_{N}\left(p_n^{\dd}\otimes
p_n\right)\right\|_{\ph} \leq \|V_n\|^2.$$

To prove the converse, observe that Theorem \ref{th_monoc}
shows that, under the stated assumptions, $\varphi\left(p_n^{\dd}\otimes p_n\right)$ is a
completely positive $\mathcal{M}p_n$-multiplier for each $n\in
\bb{N}$. Since $\{p_n^{\dd}\otimes p_n\}_{n\in \bb{N}}$ is a
covering family, we have that $\varphi$ is a completely positive
local $\mathcal{M}$-multiplier.
\end{proof}


\end{document}